\documentclass{amsart}
\usepackage{amsmath,amssymb,bm,amscd}
\usepackage[dvips]{graphicx}
\usepackage{amsmath}
\usepackage{amssymb}
\usepackage{amscd}
\usepackage{amsthm}
\usepackage[dvips]{graphicx}

\theoremstyle{definition}
\newtheorem{defn}{Definition}[section]

\newtheorem{example}[defn]{Example}
\newtheorem{rem}[defn]{Remark}
\theoremstyle{plain}
\newtheorem{thm}[defn]{Theorem}
\newtheorem{prop}[defn]{Proposition}
\newtheorem{lem}[defn]{Lemma}
\newtheorem{cor}[defn]{Corollary}

\newcommand{\Aut}{\operatorname{Aut}}
\newcommand{\Hom}{\operatorname{Hom}}
\newcommand{\id}{\operatorname{id}}

\numberwithin{equation}{section}
\title[A Khovanov type invariant derived from an unoriented HQFT]{A Khovanov type invariant derived from unoriented HQFT for links in thickened surfaces}
\author{Keiji Tagami}
\date{\today}
\address{
Department of Mathematics,
Tokyo Institute of Technology,
Oh-okayama, Meguro, Tokyo 152-8551, Japan
}
\email{tagami.k.aa@m.titech.ac.jp}
\begin{document}
\maketitle
\begin{abstract}
Two link diagrams on compact surfaces are strongly equivalent if they are related by Reidemeister moves and orientation preserving homeomorphisms of the surfaces. 
They are stably equivalent if they are related by the two previous operations and adding or removing handles. 
Turaev and Turner constructed a link homology for each stable equivalence class by applying an unoriented TQFT to a geometric chain complex similar to Bar-Natan's one. 
In this paper, by using an unoriented homotopy quantum field theory (HQFT), we construct a link homology for each strong equivalence class. 
Moreover, our homology yields an invariant of links (under ambient isotopy) in the oriented $I$-bundle of a compact surface. 
\end{abstract}
\section{Introduction}
An oriented $(d+1)$-dimensional topological quantum field theory (TQFT) \cite{atiyah1} assigns a module to each oriented closed $d$-dimensional manifold and assigns a homomorphism of modules to each oriented ($d+1$)-dimensional cobordism, satisfying certain axioms. 
If we do not assume that manifolds and cobordisms are oriented, then we call it an unoriented TQFT. 
\par
Turaev \cite{turaev:1999} defined the concept of homotopy quantum field theories (HQFTs) with target $X$, where $X$ is a connected topological space with a base point. 
An oriented HQFT assigns a module and a homomorphism of modules to each ``oriented $X$-manifold'' and ``oriented $X$-cobordism", respectively. 
An oriented $X$-manifold is a pair of an oriented 
manifold with some base points and a continuous map from the manifold to $X$. 
An oriented $X$-cobordism is a pair of an oriented cobordism with pointed boundaries and a continuous map from the cobordism to $X$. 
An unoriented HQFT is defined analogously. 
For any group $\pi$, Turaev \cite{turaev:1999} constructed a bijective correspondence between oriented ($1+1$)-dimensional HQFTs with target $X$ for $X=K(\pi, 1)$ and ``crossed $\pi$-algebras", where a crossed $\pi$-algebra $V$ is a Frobenius $\pi$-algebra endowed with a group homomorphism  $\varphi\colon\pi\rightarrow  \Aut (V)$. 
\par
The author \cite{tagami1} considered unoriented (1+1)-dimensional HQFTs with target $K(\pi, 1)$, where $\pi$ is an $\mathbf{F_{2}}$-vector space. 
For an $\mathbf{F_{2}}$-vector space $\pi$, he showed that there is a bijective correspondence between unoriented ($1+1$)-dimensional HQFTs with target $K(\pi, 1)$ and ``extended crossed $\pi$-algebras", where an extended crossed $\pi$-algebra $L$ is a crossed $\pi$-algebra endowed with a linear map $\Phi \colon L\rightarrow L$ and a family of elements $\{\theta_{\alpha} \in L_{1_{\pi}}\}_{\alpha\in\pi}$. 
\par
For each oriented link in $\mathbf{S}^{3}$, Khovanov \cite{khovanov1} defined a graded chain complex whose graded Euler characteristic equals the Jones polynomial of the link. 
Its homotopy class is a link invariant and its homology is called the Khovanov homology. 
Bar-Natan \cite{Bar-Natan-2} gave a geometric chain complex for the (original) Khovanov homology and explained it using an oriented TQFT. 
By using Bar-Natan's complex and an unoriented (1+1)-dimensional TQFT, Turaev and Turner \cite{turner-turaev:2006} constructed a link homology for each stable equivalence class (Definition~$\ref{stable}$) of link diagrams on surfaces. 
Unfortunately, their homology is not an extension of the Khovanov homology to stable equivalence classes. 
Manturov \cite{kh-virtual} gave an extension of the Khovanov complex to virtual links, that is stable equivalence classes, and 
Tubbenhauer \cite{Tubbenhauer} defined a homology for virtual links in the spirit of Bar-Natan's complex. 
Asaeda, Przytycki and Sikora \cite{link_homology_surface} also constructed link homologies for link diagrams on surfaces.  
Audoux \cite{surfaces_with_pulleys} considered ``surfaces with pulleys" and defined a link homology. 
\par
In this paper, for an oriented link diagram $D$ on an oriented compact surface $F$, we construct a link homology by using an unoriented HQFT with target $X=K(H_{1}(F;\mathbf{F_{2}}),1)$. 
The main idea is as follows: 
Firstly, we color each circle of smoothings of $D$ by the element in $H_{1}(F;\mathbf{F_{2}})$ represented by the circle. 
Since a circle labeled by an element in $\pi_{1}(X)=H_{1}(F;\mathbf{F_{2}})$ is regarded as an $X$-manifold (Remark~$\ref{color}$), we can regard each smoothing as an $X$-manifold (the manifold mapped to $X$ is the disjoint union of the circles resulting from the smoothing). 
Secondly, we use these $X$-manifolds in order to construct a geometric chain complex $([[(F,D)]]^{\ast}, \Sigma^{\ast})$ (defined in Section~$\ref{geometric_cpx}$) which is similar to Bar-Natan's  complex. 
Finally, we associate an HQFT $(A,\tau)$ to the complex and obtain a complex $(C_{(A,\tau)}^{\ast}(F,D), d_{(A,\tau)}^{\ast}):=(A([[(F,D)]]^{\ast}), \tau(\Sigma^{\ast}))$ (see Section~$\ref{link_homology2}$). 
\par
Our main result is the following theorem which is proved in Section~$\ref{link_homology2}$. 
\begin{thm}\label{homology_hqft}
Let $(F, D)$ be an oriented link diagram $D$ on an oriented compact surface $F$ and let $\pi=H_{1}(F; \mathbf{F_{2}})$. Let $(A,\tau)$ be an HQFT with target $K(\pi, 1)$ which preserves the S, T and 4-Tu relations given in Figure~$\ref{relation}$. 
We define 
\begin{align*}
H^{\ast}_{(A,\tau)}(F,D):=H(C_{(A,\tau)}^{\ast}(F,D), d_{(A,\tau)}^{\ast}). 
\end{align*}
Then $H^{\ast}_{(A,\tau)}(F,D)$ is an invariant of the oriented link in $F\times I$ represented by $(F, D)$ under ambient isotopy, where $I=[0,1]$.  
\end{thm}
\begin{rem}
Examples of such $(A, \tau)$ exist, see Sections~$\ref{link_homology1}$ and $\ref{remark}$. 
\end{rem}
\begin{rem}\label{main_rem}
Two oriented link diagrams on an oriented compact surface $F$ represent the same link in $F\times I$ if and only if they are strongly equivalent (see Definition~$\ref{stable}$). 
In order to prove Theorem~$\ref{homology_hqft}$, we show that the homology is an invariant of strong equivalence classes. 
Moreover, we will give a remark on link diagrams on a non-orientable surface and introduce that our homology is also an invariant of oriented links in an oriented $I$-bundle over a compact surface including non-orientable ones 
(see Remark~$\ref{non-orientable}$). 
\end{rem}
\par
Our construction satisfies the following duality relation which is analogous to a property of the Khovanov homology. We give a proof of Theorem~$\ref{dual}$ in Section~$\ref{section_dual}$. 
\begin{thm}\label{dual}
Let $(F,D)$ be an oriented link diagram $D$ on an oriented compact surface $F$ and let $(F,D^{!})$ be the link diagram obtained from $D$ by changing all crossings of $D$. 
For $\pi=H_{1}(F; \mathbf{F_{2}})$ and for any HQFT $(A,\tau)$ with target $K(\pi, 1)$ which preserves the S, T and 4-Tu relations given in Figure~$\ref{relation}$, there is an isomorphism of complexes 
\begin{align*}
C_{(A,\tau)}^{i}(F,D)\cong \{C_{(A,\tau)}^{i}(F,D^{!})\}^{\ast}, 
\end{align*}
where $\{C_{(A,\tau)}^{i}(F,D^{!})\}^{\ast}$ is the dual complex of $C_{(A,\tau)}^{i}(F,D^{!})$ {\rm (}the dual complex $((C^{\ast})^{i}, (d^{\ast})^{i})$ of a complex $(C^{i}, d^{i})$ over a ring $R$ is defined by $(C^{\ast})^{i}:=\Hom (C^{-i}; R)$ and $(d^{\ast})^{i}:=(d^{-i-1})^{\ast}\colon (C^{\ast})^{i}\rightarrow (C^{\ast})^{i+1}${\rm )}. 
\end{thm}
Let $D$ be a plus-adequate diagram on the plane with $n_{-}$ negative crossings. Then the $-n_{-}$-th homological degree term of the Khovanov homology of $D$ is not zero (see Proposition~$36$ in \cite{khovanov1}). 
For a special HQFT $(A,\tau)$, our homology has a similar property as follows: 
\begin{thm}\label{non_vanish}
Let $(F,D)$ be an oriented link diagram $D$ on an oriented compact surface $F$ with $n_{-}(D)$ negative crossings and let $(A, \tau)$ be the unoriented HQFT constructed in Proposition~$\ref{prophqft}$. 
Then $H^{-n_{-}(D)}_{(A,\tau)}(F,D)\neq 0$ if and only if $(F,D)$ is weak plus-adequate. 
\end{thm}
The description of a weak plus-adequate diagram is given in Definition~$\ref{condition}$. 
If a diagram is plus-adequate, the diagram is weak plus-adequate. 
We prove Theorem~$\ref{non_vanish}$ in Section~$\ref{section_nonvanish}$. 
%
For more on the minimal homological degree of Khovanov homology, see \cite{stosic2}, \cite{tagami3} and \cite{tagami2}. 
\par
This paper is organized as follows: 
In Section~$\ref{preliminaries}$, we recall the definitions of HQFTs and some algebras introduced in \cite{turaev:1999}, as well as those of unoriented HQFTs and extended crossed group algebras defined in \cite{tagami1}. 
In Sections~$\ref{geometric_cpx}$, we introduce a geometric chain complex of $X$-cobordisms. 
In Section~$\ref{link_homology}$, we establish our link homology and prove some of its properties, including Theorems~$\ref{homology_hqft}$, $\ref{dual}$ and $\ref{non_vanish}$. 
In Section~$\ref{section_example}$, we explain a combinatorial computation of our homology and some examples. 
In Section~$\ref{remark}$, we make some remarks on our homology theory.  
\par
Throughout this paper, the symbol $R$ denotes a commutative ring with unit.  
For a group $\pi$ and two elements $a,b\in\pi$, we denote the product by ``$ab$" and the unit by ``$1$". 
If $\pi$ is an abelian group, we will use ``$a+b$" and ``$0$" instead. 
In our pictures, the source of a cobordism is the bottom and the target is the top. 
\section{Unoriented HQFTs and extended crossed group algebras}\label{preliminaries}
Here we will explain terminologies used in this paper. 
\subsection{Unoriented HQFTs}
In this subsection, we recall the definition of unoriented homotopy quantum field theories (HQFTs). An oriented HQFT is introduced by Turaev \cite{turaev:1999}.  
\begin{defn}
Let $X$ be a CW-complex. We call $X$ the {\it Eilenberg-Mac Lane space of type $K(\pi$, $1)$} corresponding to a group $\pi$ ($K(\pi$, $1)$ space for short) if its homotopy group $\pi_{n}(X)=\pi$ with $n=1$ and $\pi_{n}(X)=0$ with $n\neq 1$. 
It is well known that such a CW-complex is unique up to homotopy equivalence. 
\end{defn}
\begin{defn}[\cite{turaev:1999}]
A topological space is {\it pointed} 
if each of its connected components has a base point. 
A map between pointed spaces is a continuous map preserving their base points. 
Homotopies of such maps are always supposed to be constant on the base points.
\end{defn}
\begin{rem}
In the standard definition of a pointed space, it has only one base point. 
However, in our description, a pointed space may have some base points. 
For example, a pointed space with $k$ components has $k$ base points. 
\end{rem}
\begin{defn}[\cite{turaev:1999}]
Let $X$ be a connected topological space with a base point $x_{0}\in X$. 
A pair {\rm(}$M$, $g_{M}${\rm)} is called an {\it unoriented $X$-manifold} if $M$ is a pointed closed unoriented manifold and $g_{M}$ is a continuous map from $M$ to $X$. 
We call the map $g_{M}$ the {\it characteristic map}. 
Since the spaces M and X are pointed, the map $g_{M}$ sends the base points of all components of $M$ to $x_{0}$. 
A disjoint union of unoriented $X$-manifolds and the empty set are also unoriented $X$-manifolds. 
An {\it unoriented $X$-homeomorphism} of unoriented $X$-manifolds $f\colon(M, g_{M})\rightarrow (M', g_{M'})$ is a homeomorphism from $M$ to $M'$ sending the base points of $M$ to those of $M'$ such that $g_{M}=g_{M'}\circ f$. 
\end{defn}
\par
\begin{defn}[\cite{turaev:1999}]
Let $X$ be a connected topological space with a base point $x_{0}\in X$. A triple {\rm(}$W, M_{0}, M_{1}${\rm)} is called an {\it unoriented cobordism} when $W$ is a compact manifold whose boundary is the disjoint union of pointed closed manifolds $M_{0}$ and $M_{1}$. 
An {\it unoriented $X$-cobordism} is a tuple {\rm(}$W, M_{0}, M_{1}, g${\rm)} such that the triple {\rm(}$W, M_{0}, M_{1}${\rm)} is an unoriented cobordism and that $g\colon W\rightarrow X$ is a  continuous map which sends the base points of $M_{0}$ and $M_{1}$ to $x_{0}\in X$. 
We call the boundary $M_{0}$ the {\it bottom base}, $M_{1}$ the {\it top base} and the map $g$ the {\it characteristic map}. 
An {\it unoriented $X$-homeomorphism} of $X$-cobordisms $f\colon (W, M_{0}, M_{1}, g)\rightarrow (W', M'_{0}, M'_{1}, g')$ is a homeomorphism from $W$ to $W'$ inducing unoriented $X$-homeomorphisms $M_{0}\rightarrow M'_{0}$ and $M_{1}\rightarrow M'_{1}$ such that $g=g'\circ f$. 
\end{defn}
\begin{defn}[\cite{turaev:1999}]\label{HQFT}
Fix an integer $d \geq 0$ and a path connected topological space $X$ with a base point $x_{0} \in X$. An {\it unoriented {\rm(}$d+1${\rm)}-dimensional homotopy quantum field theory} {\rm(}HQFT for short {\rm)} {\rm(}$A, \tau${\rm)} over $R$ with target $X$ assigns 
\begin{itemize}
\item a finitely generated projective $R$-module $A(M, g)$ {\rm(}$A(M)$ for short{\rm)} to any unoriented $d$-dimensional $X$-manifold $(M, g)$, 
\item an $R$-isomorphism $f_{\sharp}\colon A(M,g)\rightarrow A(M', g')$ to any unoriented $X$-homeomorphism of $d$-dimensional $X$-manifolds $f\colon (M, g)\rightarrow (M', g')$, 
\item an $R$-homomorphism $\tau(W, g)\colon A(M_{0}, g|_{M_{0}})\rightarrow A(M_{1}, g|_{M_{1}})$ to any {\rm(}$d+1${\rm)}-dimensional $X$-cobordism {\rm(}$W, M_{0}, M_{1}, g${\rm)}. 
\end{itemize}
Moreover these modules and homomorphisms should satisfy the following axioms: 
\par
$(1)$ for unoriented $X$-homeomorphisms of unoriented $X$-manifolds $f\colon M\rightarrow M'$ and $f'\colon M'\rightarrow M''$, we have $(f'\circ f)_{\sharp}=f'_{\sharp}\circ f_{\sharp}$, 
\par
$(2)$ for unoriented $d$-dimensional $X$-manifolds $M$ and $N$, there is a natural isomorphism $A(M\sqcup N)=A(M)\otimes A(N)$, where $M\sqcup N$ is the disjoint union of $M$ and $N$, 
\par
$(3)$ $A(\emptyset)=R$, 
\par
$(4)$ for any unoriented $X$-cobordism $W$, the homomorphism $\tau(W)$ is natural with respect to unoriented $X$-homeomorphisms, 
that is, for any unoriented $X$-homeomorphism $f\colon (W,M_{0},M_{1},g)\rightarrow (W',M'_{0},M'_{1},g')$ of $X$-cobordisms, the following diagram is commutative: 
\begin{equation*}
\begin{CD}
A(M_{0})  @>(f|_{M_{0}})_{\sharp}>> A(M'_{0}) \\
@VV\tau(W)V    @VV\tau(W')V  \\
A(M_{1})  @>(f|_{M_{1}})_{\sharp}>>  A(M'_{1}) 
\end{CD}
\end{equation*}
\par
$(5)$ if an unoriented $(d+1)$-dimensional $X$-cobordism $(W, M_{0}, M_{1}, g)$ is the disjoint union of two unoriented $(d+1)$-dimensional $X$-cobordisms $W_{0}$ and $W_{1}$, then $\tau(W)=\tau(W_{1})\otimes \tau(W_{0})$, 
\par
$(6)$ if an oriented $(d+1)$-dimensional $X$-cobordism $(W, M_{0}, M_{1}, g)$ is obtained from two $(d+1)$-dimensional $X$-cobordisms $(W_{0}, M_{0}, N)$ and $(W_{1}, N', M_{1})$ by gluing along $f\colon N\rightarrow N'$, then $\tau(W)=\tau(W_{1})\circ f_{\sharp} \circ \tau(W_{0})$, 
\par
$(7)$ for any unoriented $d$-dimensional $X$-manifold {\rm(}$M, g${\rm)} and for any continuous map $F\colon M\times [0,1]\rightarrow X$ such that $F|_{M\times 0}=F|_{M\times 1}=g$ and that $F(\{m\}\times [0,1])=\{x_{0}\}$ for any base point $m$ of $M$, we have $\tau(M\times [0,1], M\times 0, M\times 1, F)=\id_{A(M)}\colon A(M)\rightarrow A(M)$, 
\par
$(8)$ for any unoriented $(d+1)$-dimensional X-cobordism {\rm(}$W, g${\rm)}, the map $\tau(W)$ is preserved under any homotopy of $g$ relative to $\partial{W}$. 
\end{defn}
\begin{rem}[\cite{turaev:1999}]\label{rel_boud}
If $f$ and $f'\colon M\rightarrow X$ are homotopic, there is a natural isomorphism $\tau(M\times [0,1], F)\colon A(M, f)\cong A(M, f')$, where $F$ is the homotopy. Hence we can suppose that $A(M, f)$ is preserved under any homotopy of $f$. Similarly $\tau(W, g)$ is preserved under any homotopy of $g$ (may not be relative to $\partial{W}$). 
\end{rem}
\begin{rem}[\cite{tagami1}]\label{color}
Suppose that $X=K(\pi, 1)$ with an $\mathbf{F_{2}}$-vector space $\pi$. 
Let $\mathbf{S}^{1}$ be an unoriented circle and $g\colon \mathbf{S}^{1}\rightarrow X$ a continuous map. 
Then we can regard the homotopy class of $g$ as an element $\alpha\in\pi$ since $X=K(\pi, 1)$ and $\pi$ is an $\mathbf{F_{2}}$-vector space. 
In the sense of Remark~$\ref{rel_boud}$, we denote the unoriented $X$-manifold $(\mathbf{S}^{1}, g)$ by $(\mathbf{S}^{1}, \alpha)$. 
\par
Analogously, we represent an unoriented ($1+1$)-dimensional $X$-cobordism as an unoriented  ($1+1$)-dimensional cobordism with some arcs and loops labeled by elements in $\pi$. 
For example, see the $X$-cobordism depicted in Figure~$\ref{x-cobordism}$. Its top base is an $X$-manifold $(\mathbf{S}^{1}, \alpha+\beta)$ and its bottom base is the disjoint union of two $X$-manifolds $(\mathbf{S}^{1}, \alpha)$ and $(\mathbf{S}^{1}, \beta)$. 
Its characteristic map sends each labeled arc or loop to the loop on $X$ corresponding to the label. 
Since $X=K(\pi, 1)$, such a characteristic map is uniquely determined up to homotopy. 
\begin{figure}[!h]
\begin{center}
\includegraphics[scale=0.2]{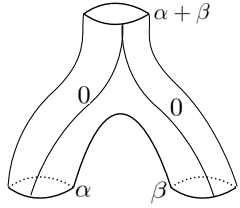}
\end{center}
\caption{Example of $X$-cobordism. }
\label{x-cobordism}
\end{figure}
\end{rem}
\subsection{Extended crossed group algebras}
In this subsection, we recall some algebras which are introduced in \cite{turaev:1999} and \cite{tagami1}. 
\begin{defn}
An $R$-algebra $L$ is a {\it $\pi$-algebra} over $R$ if $L$ is an associative algebra over $R$ endowed with a splitting $L=\bigoplus_{\alpha  \in \pi}L_{\alpha}$ such that each $L_{\alpha}$ is a finitely generated projective $R$-module, that $L_{\alpha}L_{\beta}\subset L_{\alpha \beta }$ for any $\alpha, \beta \in\pi$, and that $L$ has the unit element $1_{L}\in L_{1}$. 
\end{defn}
Let $V$ and $W$ be $R$-modules and $\eta\colon V\otimes W\rightarrow R$ a bilinear form. The map $\eta$ is non-degenerate if the two maps $d\colon V\rightarrow \Hom_{R}(W, R)$ defined by $d(v)(w)=\eta(v, w)$ and $s\colon W\rightarrow \Hom_{R}(V, R)$ defined by $s(w)(v)=\eta(v, w)$ are isomorphisms, where $v\in V$ and $w\in W$. 
\begin{defn}[\cite{turaev:1999}]\label{frobenius}
A pair $(L, \eta)$ is a {\it Frobenius $\pi$-algebra} over $R$ if $L$ is a $\pi$-algebra over $R$ and $\eta \colon L_{\alpha }\otimes L_{\beta}\rightarrow R $ is an $R$-bilinear form such that 
\par
$(1)$
$\eta(L_{\alpha}\otimes L_{\beta})=0$ if $\alpha \beta \neq 1$ and the restriction of $\eta$ to $L_{\alpha }\otimes L_{\alpha^{-1} }$ is non-degenerate for any $\alpha \in\pi$, 
\par
$(2)$
$\eta(ab, c)=\eta(a, bc)$ for any $a, b, c\in L$. 
\end{defn}
A Frobenius $\pi$-algebra with a trivial group $\pi$ is called a {\it Frobenius algebra} \cite{abrams:1996}. 
\par
For any Frobenius $\pi$-algebra $(L, \eta)$, we denote the group by $\Aut(L)$ which consists of algebra automorphisms preserving $\eta$. 
\begin{defn}[\cite{turaev:1999}]\label{cross}
A triple $(L, \eta, \varphi)$ is a {\it crossed $\pi$-algebra} over $R$ if the pair $(L, \eta)$ is a Frobenius $\pi$-algebra over $R$ and $\varphi\colon \pi\rightarrow\Aut(L)$ is a group homomorphism satisfying the following axioms:
\par
$(1)$
for any $\beta\in\pi$, the map $\varphi_{\beta}:=\varphi(\beta)$ is an algebra automorphism of $L$ preserving $\eta$ and satisfying $\varphi_{\beta}(L_{\alpha})\subset L_{\beta \alpha \beta ^{-1}}$ for any $\alpha \in\pi$, 
\par
$(2)$
$\varphi_{\alpha}|_{L_{\alpha}}=\id_{L_{\alpha}}$ for any $\alpha \in\pi$, 
\par
$(3)$
for any $a\in L_{\alpha}$ and $b\in L_{\beta}$, we have $\varphi_{\beta}(a)b=ba$, 
\par
$(4)$
for any $\alpha, \beta\in\pi$ and any $c\in L_{\alpha \beta \alpha^{-1} \beta^{-1} }$, we have $\operatorname {Tr}(c\varphi_{\beta}\colon L_{\alpha }\rightarrow L_{\alpha })=\operatorname {Tr}(\varphi_{\alpha^{-1}}c\colon L_{\beta}\rightarrow L_{\beta})$, where {\rm Tr} is the $R$-valued trace of endmorphisms of finitely generated projective $R$-modules (see for instance \cite{quantum-turaev}). 
\end{defn}
Turaev \cite{turaev:1999} showed that there exists a bijective correspondence between oriented HQFTs with target $K(\pi, 1)$ space and crossed $\pi$-algebras. 
\begin{thm}[Theorem~$4.1$ in {\cite{turaev:1999}}]\label{turaev}
Let $\pi$ be a group and $X$ the $K(\pi, 1)$ space. Then every oriented $(1+1)$-dimensional HQFT with target $X$ over $R$ determines an underlying crossed $\pi$-algebra over $R$. This induces a bijection between the set of isomorphism classes of oriented $(1+1)$-dimensional HQFTs and the set of isomorphism classes of crossed $\pi$-algebras. 
\end{thm}
%
%
Now we define extended crossed group-algebras. 
\begin{defn}\label{extcross}
Let $\pi$ be a group such that $\alpha ^{2}=1$ for any $\alpha \in\pi$ (such a group is an $\mathbf{F_{2}}$-vector space). A tuple $(L, \eta, \varphi, \{\theta_{\alpha}\}_{\alpha\in\pi}, \Phi)$ is an {\it extended crossed $\pi$-algebra} over $R$ if the triple $(L, \eta, \varphi)$ is a crossed $\pi$-algebra over $R$, and a family of elements $\{\theta_{\alpha}\in L_{1}\}_{\alpha\in\pi}$ and an $R$-linear map of $R$-modules $\Phi\colon L\rightarrow L$ satisfy the following axioms:
\par
$(1)$
$\Phi^{2}=\id$, 
\par
$(2)$
$\Phi(L_{\alpha})\subset L_{\alpha}$ for any $\alpha\in\pi$, 
\par
$(3)$
for any $v, w\in L$, we have $\Phi(vw)=\Phi(w)\Phi(v)$, 
\par
$(4)$
$\Phi(1_{L})=1_{L}$, 
\par
$(5)$
$\eta\circ(\Phi\otimes \Phi)=\eta$, 
\par
$(6)$
for any $\alpha \in\pi$, we have $\Phi\circ\varphi_{\alpha}=\varphi_{\alpha}\circ\Phi$, 
\par
$(7)$
for any $\alpha, \beta, \gamma\in\pi$ and $v\in L_{\alpha \beta }$, we have 
\begin{align*}
m\circ(\Phi\otimes \varphi_{\gamma })\circ \Delta _{\alpha, \beta }(v)&=\varphi_{\gamma }(\theta_{\alpha \gamma }\theta_{\gamma }v),\\ 
m\circ(\varphi_{\gamma }\otimes \Phi)\circ \Delta _{\alpha, \beta }(v)&=\varphi_{\gamma }(\theta_{\beta \gamma }\theta_{\gamma }v), 
\end{align*}
where $\Delta _{\alpha, \beta }\colon L_{\alpha \beta }\rightarrow L_{\alpha}\otimes L_{\beta}$ is defined by the following relation: 
\begin{align}
(\id\otimes\eta)\circ(\Delta_{\alpha , \beta }\otimes\id)=m \label{eq2-1}
\end{align}
(since $\eta$ is non-degenerate and each $L_{\alpha }$ is finitely generated, such a map $\Delta_{\alpha ,\beta } $ is uniquely determined). 
\par
$(8)$
for any $\alpha, \beta \in\pi$ and $v\in L_{\alpha }$, we have $\Phi(\theta_{\beta}v_{\alpha})=\varphi_{\beta\alpha}(\theta_{\beta\alpha}v_{\alpha})$, 
\par
$(9)$
for any $\alpha\in\pi$, we have $\Phi(\theta_{\alpha})=\theta_{\alpha}$, 
\par
$(10)$
for any $\alpha, \beta\in\pi$, we have $\varphi_{\beta}(\theta_{\alpha})=\theta_{\alpha}$, 
\par
$(11)$ for any $\alpha, \beta, \gamma\in\pi$, we have $\theta_{\alpha }\theta_{\beta }\theta_{\gamma }=q(1)\theta_{\alpha \beta \gamma }$, where $q\colon R\rightarrow L_{1}$ is defined as follows: 
Let $\{a_{i}\in L_{\alpha \beta }\}_{i=1}^{n}$ and $\{b_{i}\in L_{\alpha \beta }\}_{i=1}^{n}$ be families of elements of $L_{\alpha \beta }$ satisfying 
\begin{align}
\sum_{i}\eta(b_{i}\otimes v)a_{i}=\varphi_{\beta \gamma }(v) \label{eq2-2}
\end{align}
for any $v\in L_{\alpha \beta }$ (from the same reason as ($7$), such $a_{i}$ and $b_{i}$ are unique). Then, we put $q(1):=\sum_{i}a_{i}b_{i}$.  
\end{defn}
The author \cite{tagami1} proved the following result. 
\begin{thm}[Theorem~$3.11$ in \cite{tagami1}]\label{thm1}
Let $\pi$ be an $\mathbf{F_{2}}$-vector space and $X$ the $K(\pi, 1)$ space. Then every unoriented $(1+1)$-dimensional HQFT with target $X$ over $R$ determines an underlying extended crossed $\pi$-algebra over $R$. This induces a bijection between the set of isomorphism classes of unoriented $(1+1)$-dimensional HQFTs over $R$ and the set of isomorphism classes of extended crossed $\pi$-algebras over $R$. 
\end{thm}
\section{Complexes of $X$-cobordisms for link diagrams on surfaces}\label{geometric_cpx}
In this section, we construct a geometric chain complex of $X$-cobordisms for link diagrams on oriented surfaces (Proposition~$\ref{mainprop}$). 
To define the complex, we use notations in \cite{turner-turaev:2006} as a reference. 
\begin{defn}\label{stable}
Let $(F,D)$ and $(F',D')$ be link diagrams $D$ and $D'$ on oriented surfaces $F$ and $F'$, respectively. 
Then $(F, D)$ and $(F', D')$ are {\it strongly equivalent} if they are related by the following two operations: 
\begin{itemize} 
\item finite sequence of Reidemeister moves on the surfaces, 
\item orientation preserving homeomorphisms of the surfaces. 
\end{itemize}
Suppose that $F$ and $F'$ are closed. Then, $(F,D)$ and $(F',D')$ are {\it stably equivalent} if they are related by the above two operations and the following: 
\begin{itemize}
\item adding or removing handles which do not affect the link diagram.  
\end{itemize}
\end{defn}
For a link diagram on a surface, we will define a complex by using the construction of Bar-Natan's complex given in \cite{Bar-Natan-2}. 
\begin{defn}
Let $F$ be an oriented compact surface and $X$ the Eilenberg-MacLane space of type $K(H_{1}(F; \mathbf{F_{2}}), 1)$. 
Define $\mathcal{UX}Cob(F)$ to be the following category. 
The objects of $\mathcal{UX}Cob(F)$ are collections of disjoint (unoriented) closed $1$-dimensional $X$-manifolds ($\Gamma$, $\alpha$) in $F$. 
A morphism $(\Gamma, \alpha)\rightarrow (\Gamma', \alpha')$ is an (unoriented) $(1+1)$-dimensional $X$-cobordism $\Sigma$ embedded in $F\times [0,1]$ which satisfies 
$\partial\Sigma=\Sigma \cap (F\times \{0,1\})$, $\Sigma \cap (F\times \{0\})=(\Gamma, \alpha)$ and $\Sigma \cap (F\times \{1\})=(\Gamma', \alpha')$. 
Two morphisms are identified if they are related by an isotopy which preserves boundaries and characteristic maps. 
\end{defn}
%
%
\begin{defn}
Define $\mathcal{UX}Cob(F)_{/r}$ to be the category obtained from $\mathcal{UX}Cob(F)$ by dividing the set of the morphisms by the equivalent relation generated by the S, T and $4$-Tu relations (see Figure~$\ref{relation}$). 
\begin{figure}[!h]
\begin{center}
\includegraphics[scale=0.65]{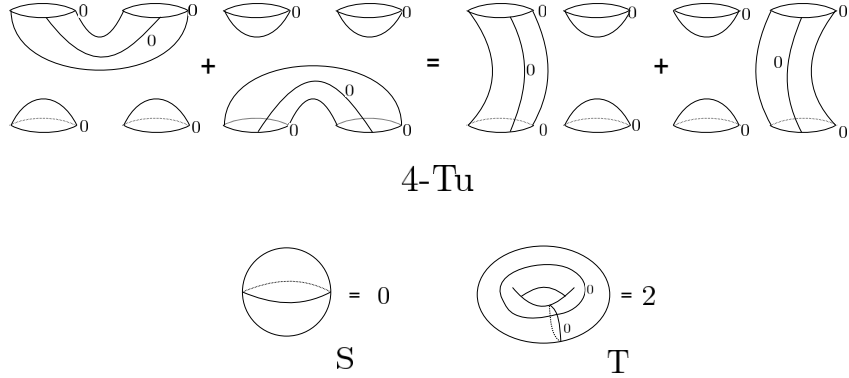}
\end{center}
\caption{Bar-Natan's relation. 
In these cobordisms, their characteristic maps send to the labeled arcs to the loops in $X$ corresponding to the labels. Such maps are uniquely determined up to homotopy since $X$ is the $K(H_{1}(F;\mathbf{F_{2}}), 1)$ space. }
\label{relation}
\end{figure}
\end{defn}
\begin{defn}
Let $\mathcal{C}$ be an additive category. 
Then we define a category $\mathcal{M}at(\mathcal{C})$ as follows: 
The objects of $\mathcal{M}at(\mathcal{C})$ are finite families $\{ \mathcal{C}_{i}\in\mathcal{C}\}$. 
For convenience, we denote $\{\mathcal{C}_{i}\in\mathcal{C}\}$ by $\bigoplus_{i}\mathcal{C}_{i}$. 
A morphism $\Sigma\colon\bigoplus_{i}\mathcal{C}_{i}\rightarrow \bigoplus_{i}\mathcal{C}_{j}'$ is a matrix $\Sigma=(\Sigma_{i}^{j})_{i,j}$ of morphisms $\Sigma_{i}^{j}\colon\mathcal{C}_{i}\rightarrow \mathcal{C}_{j}'$. If $\mathcal{C}$ is not additive category, then made it so by allowing formal $\mathbf{Z}$-linear combinations of morphisms and define $\mathcal{M}at(\mathcal{C})$ as above. 
\end{defn}
For an oriented link diagram $D$ on an oriented compact surface $F$, we will define a complex $([[(F,D)]]^{\ast}, \Sigma^{\ast})$ as follows: 
%
Let $(F,D)$ be an oriented link diagram $D$ on an oriented compact surface $F$. Fix an order of the crossings of $D$. 
For each crossings of $D$, we define $0$-smoothing and $1$-smoothing as in Figure~$\ref{smoothing}$. 
A smoothing of $D$ is a diagram where each crossing is changed by either $0$-smoothing or $1$-smoothing. 
Let $n$ be the number of the crossings of $D$. Then $D$ has $2^{n}$ smoothings.  
By using given ordering of $D$, we obtain a natural bijection between the set of the smoothings of $D$ and the set $\{0, 1\}^{n}$, that is, 
to any $\varepsilon=(\varepsilon_{1}, \dots, \varepsilon_{n})\in\{0,1\}^{n}$, we associate the smoothing $D_{\varepsilon}$ where the $i$-th crossing is $\varepsilon_{i}$-smoothed. 
Each smoothing $D_{\varepsilon}$ is a collection of disjoint circles. 
We can regard a smoothing $D_{\varepsilon}$ as an unoriented closed $1$-dimensional $X$-manifold by assigning the homology class $\alpha=[\Gamma]\in H_{1}(F; \mathbf{F_{2}})=\pi_{1}(X)$ to each circle $\Gamma$ of $D_{\varepsilon}$ (see Remark~$\ref{color}$). 
Then define 
\begin{align*}
[[(F,D)]]^{i}:=\bigoplus_{|\varepsilon|=i+n_{-}}D_{\varepsilon} \in Ob(\mathcal{M}at(\mathcal{UX}Cob(F)_{/r})), 
\end{align*}
where $|\varepsilon|=\sum_{i=1}^{n}\varepsilon_{i}$ and $n_{-}$ is the number of the negative crossings of $D$. 
\begin{figure}[!h]
\begin{center}
\includegraphics[scale=0.55]{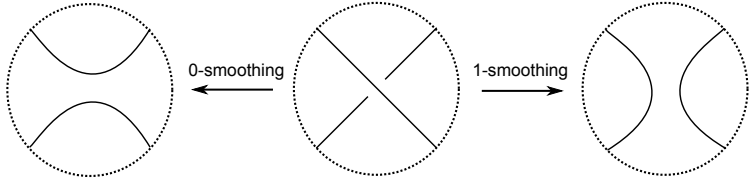}
\end{center}
\caption{$0$-smoothing and $1$-smoothing. }
\label{smoothing}
\end{figure}
\par
The morphism $\Sigma^{i}\colon[[(F,D)]]^{i}\rightarrow [[(F,D)]]^{i+1}$ is defined as follows: 
Take elements $\varepsilon, \varepsilon'\in\{0,1\}^{n}$ such that $\varepsilon_{j}=0$ and $\varepsilon_{j}'=1$ for some $j$ and that $\varepsilon_{i}=\varepsilon_{i}'$ for all $i\neq j$. 
For such a pair $(\varepsilon, \varepsilon')$, 
the smoothing $D_{\varepsilon}$ can be identical to $D_{\varepsilon'}$ out side a disc. 
We call this disc the changing disc (see p. 1085 in \cite{turner-turaev:2006}). 
Let $B$ be the changing disc. 
Define a cobordism $\Sigma_{\varepsilon\rightarrow \varepsilon'}\subset F\times [0,1]$ to be the $X$-surface obtained from $D_{\varepsilon}\times [0,1]$ by replacing $B\times [0,1]$ with the saddle depicted in Figure~$\ref{saddle}$ in such a way that the bottom base of $\Sigma_{\varepsilon\rightarrow \varepsilon'}$ is $D_{\varepsilon}$ and the top is $D_{\varepsilon'}$. 
\par
If there exist distinct integers $i$ and $j$ such that $\varepsilon_{i}\neq\varepsilon'_{i}$ and $\varepsilon_{j}\neq\varepsilon'_{j}$, then define $\Sigma_{\varepsilon\rightarrow \varepsilon'}=0$. 
\par
In this setting, we define a map $\Sigma^{i}\colon[[(F,D)]]^{i}\rightarrow [[(F,D)]]^{i+1}$ by the matrix $((-1)^{l(\varepsilon, \varepsilon')}\Sigma_{\varepsilon \rightarrow \varepsilon'})_{\varepsilon, \varepsilon'}$. 
Here $|\varepsilon|=i$, $|\varepsilon'|=i+1$ and $l(\varepsilon, \varepsilon')$ is the number of $1$'s in front of (in our order) the factor of $\varepsilon$ which is different from $\varepsilon'$. 
\begin{figure}[!h]
\begin{center}
\includegraphics[scale=0.29]{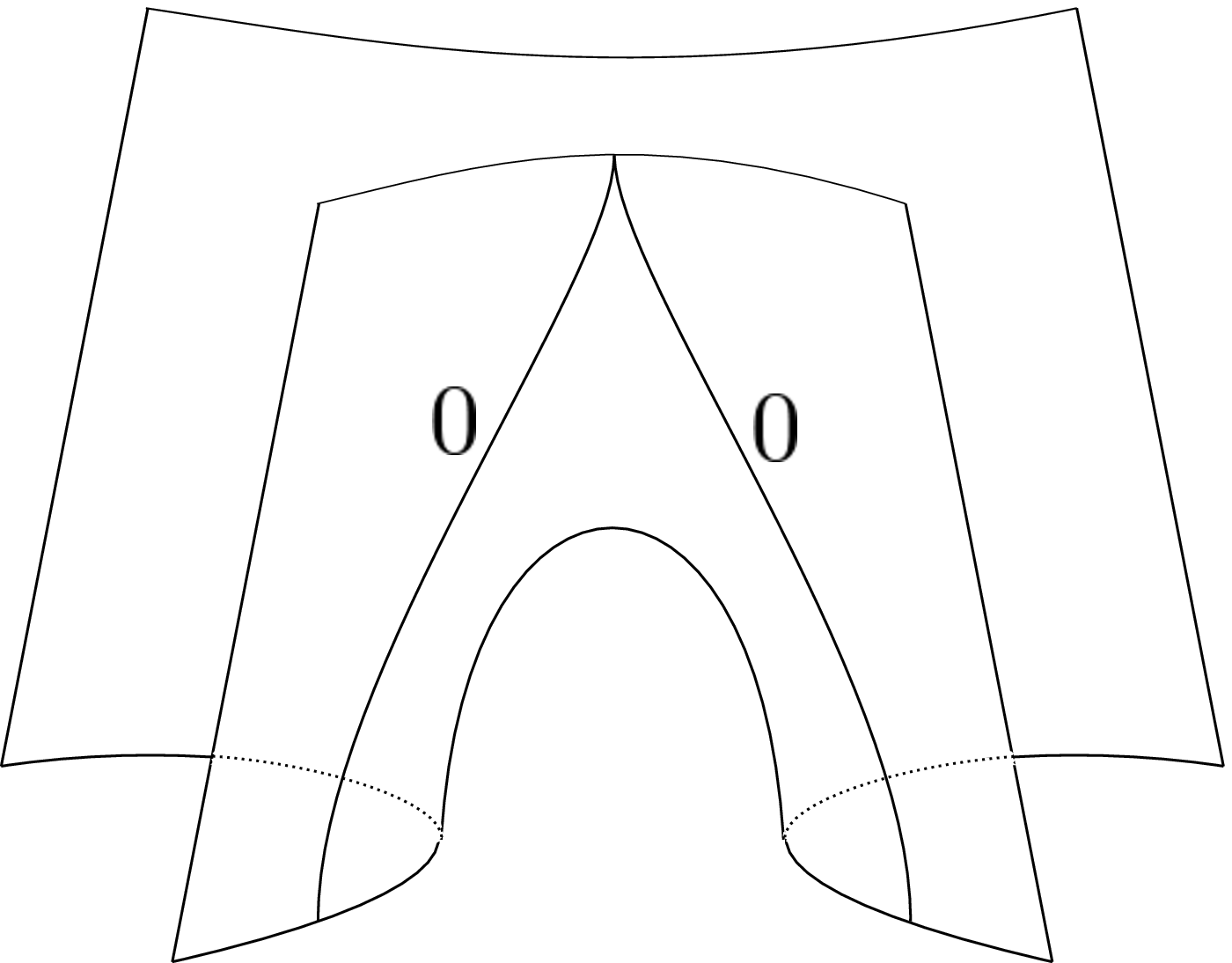}
\end{center}
\caption{Saddle. }
\label{saddle}
\end{figure}
\begin{lem}\label{lem:cpx}
For any oriented link diagram $D$ on an oriented compact surface $F$, the pair $([[(F,D)]]^{\ast}, \Sigma^{\ast})$ is a chain complex and its homotopy class does not depend on the choice of an order of the crossings of $D$. 
\end{lem}
\begin{proof}
This follows from the same discussion in Proposition~$3.1$ in \cite{turner-turaev:2006}. 
\end{proof}
From the following proposition, the complex $([[(F,D)]]^{\ast}, \Sigma^{\ast})$ is an invariant of strong equivalence. 
\begin{prop}\label{mainprop}
Let $D$ be an oriented link diagram on an oriented compact surface $F$. The homotopy class of the complex $([[(F,D)]]^{\ast}, \Sigma^{\ast})$ in $\mathcal{M}at(\mathcal{UX}Cob(F)_{/r})$ is an invariant under Reidemeister moves. 
\end{prop}
\begin{proof}
In Proposition~$3.3$ in \cite{turner-turaev:2006}, Turner and Turaev gave an analogous statement for link diagrams on surfaces by using Bar-Natan's method. 
We can apply the same proof to our setting. 
\end{proof}
\section{link homologies over $\mathbf{F_{2}}$}\label{link_homology}
In this section, we construct our link homology and show that it has some properties similar to the Khovanov homology. \par
In Section~$\ref{link_homology2}$, we apply an HQFT which preserves the S, T and 4-Tu relations to the geometric chain complex given in Section~$\ref{geometric_cpx}$ in order to construct our link homology, 
and we prove Theorem~$\ref{homology_hqft}$. 
In Section~$\ref{section_dual}$, we give a duality relation of our link homology (Theorem~$\ref{dual}$). 
In Sections~$\ref{link_homology1}$ and $\ref{section_nonvanish}$, we construct an HQFT over $\mathbf{F_{2}}$ with target $K(\pi, 1)$ for an $\mathbf{F_{2}}$-vector space $\pi$ and give a non-vanishing condition of the link homology (Theorem~$\ref{non_vanish}$). 
%
%
%
%
\subsection{Applying an unoriented HQFT to geometric chain complexes}\label{link_homology2}
In this subsection, we construct our link homology. 
\par
Let F be an oriented compact surface. 
Define $\pi:=H_{1}(F;\mathbf{F_{2}})$. 
Let $(A, \tau)$ be an HQFT with target $K(\pi,1)$ which preserves the S, T and 4-Tu relations given in Figure~$\ref{relation}$. 
We can regard $(A,\tau)$ as a functor 
\begin{center}
$\mathcal{UX}Cob(F)\rightarrow Proj_{R}$ 
\end{center}
by $(\Gamma, \alpha)\mapsto A(\Gamma, \alpha)$ and $W\mapsto \tau(W)$, where $Proj_{R}$ is the category of projective $R$-modules. 
By taking formal direct sums in $\mathcal{M}at(\mathcal{UX}Cob(F))$ to direct sums in $Proj_{R}$, the HQFT $(A,\tau)$ extends to a functor 
\begin{center}
$\mathcal{M}at(\mathcal{UX}Cob(F))\rightarrow Proj_{R}$. 
\end{center}
Moreover, the HQFT $(A,\tau)$ defines a functor
\begin{center}
$\mathcal{M}at(\mathcal{UX}Cob(F)_{/r})\rightarrow Proj_{R}$. 
\end{center}
For an oriented link diagram $(F,D)$, we define 
\begin{align*}
C_{(A,\tau)}^{i}(F,D)&:=A([[(F,D)]]^{i})\\
d_{(A,\tau)}^{i}(F,D)=d_{(A,\tau)}^{i}&:=\tau(\Sigma^{i})\colon C_{(A,\tau)}^{i}(F,D)\rightarrow C_{(A,\tau)}^{i+1}(F,D). 
\end{align*} 
From Lemma~$\ref{lem:cpx}$, $(C_{(A,\tau)}^{i}(F,D), d_{(A,\tau)}^{i})$ is a chain complex. 
Denote its homology by $H^{\ast}_{(A,\tau)}(F,D)$. 
By the following lemma, the homology $H^{\ast}_{(A,\tau)}(F,D)$ gives an invariant of strong equivalence. 
\begin{lem}\label{mainlem}
The homotopy class of the complex $(C_{(A,\tau)}^{i}(F,D), d_{(A,\tau)}^{i})$ is an invariant of strong equivalence classes. 
In particular, its homology $H^{\ast}_{(A,\tau)}(F,D)$ is an invariant of strong equivalence classes. 
\end{lem}
\begin{proof}
It follows from Proposition~$\ref{mainprop}$ that the homotopy class of the complex $(C_{(A,\tau)}^{i}(F,D), d_{(A,\tau)}^{i})$ is preserved under Reidemeister moves. 
\par
Let $f\colon (F,D)\rightarrow (F',D')$ be an orientation preserving homeomorphism. 
This induces $X$-homeomorphisms on the smoothings. 
Applying the HQFT, we obtain the isomorphism $f_{\sharp}\colon C_{(A,\tau)}^{\ast}(F,D)\rightarrow C_{(A,\tau)}^{\ast}(F',D')$.  
In order to check that $f_{\sharp}$ is a chain map, consider the homeomorphism $\tilde{f}\colon F\times I\rightarrow F'\times I$ given by $(s,t)\mapsto (f(s),t)$. 
The map $\tilde{f}$ induces $X$-homeomorphisms $\Sigma_{\varepsilon\rightarrow \varepsilon'}\rightarrow \Sigma'_{\varepsilon\rightarrow \varepsilon'}$ on the $X$-cobordisms defining the differential maps of the complexes. 
By axiom $(4)$ in Definition~$\ref{HQFT}$, we have $f_{\sharp}\circ\tau(\Sigma_{\varepsilon\rightarrow \varepsilon'})=\tau(\Sigma'_{\varepsilon\rightarrow \varepsilon'})\circ f_{\sharp}$. 
This means $f_{\sharp}$ is a chain map. 
Hence $f_{\sharp}\colon (C_{(A,\tau)}^{i}(F,D), d_{(A,\tau)}^{i})\rightarrow (C_{(A,\tau)}^{i}(F',D'), d_{(A,\tau)}^{i})$ is an isomorphism of the complexes. 
\end{proof}
\begin{proof}[{\bf Proof of Theorem~$\ref{homology_hqft}$}]
Theorem~$\ref{homology_hqft}$ immediately follows from Lemma~$\ref{mainlem}$ and Remark~$\ref{main_rem}$. 
\end{proof}
\subsection{A duality of our link homology}\label{section_dual}
In this subsection, we prove that our homology has a duality relating the homology groups of a link diagram to those of its mirror image (Theorem~$\ref{dual}$). \par
Let $\pi$ be an $\mathbf{F_{2}}$-vector space with finite dimension and $(L, \eta, \varphi, \{\theta_{\alpha}\}_{\alpha\in\pi}, \Phi)$ an extended crossed $\pi$-algebra. 
Define $\tilde{L}:=\Hom(L; \mathbf{F_{2}})$. 
We give $\tilde{L}$ an  extended crossed $\pi$-algebra structure as follows: 
\begin{itemize}
\item the unit $\tilde{1}$ is $\eta(1\otimes \cdot)\in\tilde{L}$, 
\item the multiplication $\tilde{m}\colon\tilde{L}\otimes \tilde{L}\rightarrow \tilde{L}$ is given by $\tilde{m}(u^{\ast}\otimes v^{\ast})(s):=(v^{\ast}\otimes u^{\ast})(\Delta (s))$, where $u^{\ast},v^{\ast}\in\tilde{L}$, $s\in L$ and $\Delta$ is the comultiplication of $L$, 
\item the inner product $\tilde{\eta} \colon \tilde{L}\otimes \tilde{L}\rightarrow \mathbf{F_{2}}$ is given by $\tilde{\eta}(u^{\ast},v^{\ast}):=(u^{\ast}\otimes v^{\ast})(\Delta (1))$, where $u^{\ast},v^{\ast}\in\tilde{L}$ and $1\in L_{0}$ is the unit of $L$, 
\item the element $\tilde{\theta}_{\alpha}$ is $\eta(\theta_{\alpha}\otimes \cdot)\in\tilde{L}$, 
\item the map $\tilde{\varphi}:=\varphi^{\ast}$ for any $\alpha\in\pi$, where $\varphi^{\ast}$ is the dual of $\varphi$, 
\item the map $\tilde{\Phi}:=\Phi^{\ast}$, where $\Phi^{\ast}$ is the dual of $\Phi$.  
\end{itemize}
Indeed, the algebra $\tilde{L}$ is extended crossed $\pi$-algebra. For example, the map $\tilde{m}$ is associative (see Figure~$\ref{dualpicture}$), the element $\tilde{1}$ is the unit (see Figure~$\ref{dualpicture2}$) and so on. 
\begin{lem}\label{duallem}
The map $\Psi \colon L\rightarrow \tilde{L}$ given by $\Psi(v):=\eta(v\otimes \cdot)$ is an isomorphism of extended crossed $\pi$-algebras. 
\end{lem}
\begin{proof}
The map $\Psi $ is bijective since $\eta$ is non-degenerate (see Definition~$\ref{frobenius}$). 
From the definition of $\tilde{L}$, we have $\Psi (1)=\tilde{1}$ and $\Psi (\theta_{\alpha})=\tilde{\theta_{\alpha}}$. 
Moreover, from Definitions~$\ref{cross}$ $(1)$ and $\ref{extcross}$ $(5)$, we obtain $\Psi \circ\Phi =\tilde{\Phi }\circ\Psi $ and $\Psi \circ\varphi _{\alpha} =\tilde{\varphi _{\alpha} }\circ\Psi $. 
Furthermore, from Figures~$\ref{dualpicture3}$ and $\ref{dualpicture4}$, the map $\Psi $ satisfies the following:  
\begin{enumerate}
\item $\tilde{m}(\Psi (u)\otimes\Psi (v))=\tilde{m}(\eta(u\otimes\cdot)\otimes\eta(v\otimes\cdot))=(\eta(v\otimes\cdot)\otimes\eta(u\otimes\cdot))\circ \Delta=\eta(uv\otimes\cdot)=\Psi (uv)$ for any $u,v\in L$, \label{m}
\item $\tilde{\eta}(\Psi (u)\otimes\Psi (v))=\tilde{\eta}(\eta(u\otimes \cdot)\otimes \eta(v\otimes \cdot))=\eta(u\otimes v)$ for any $u,v\in L$. \label{eta}
\end{enumerate}
These mean $\Psi $ is an isomorphism of extended crossed $\pi$-algebras. 
\begin{figure}[!h]
\begin{center}
\includegraphics[scale=0.4]{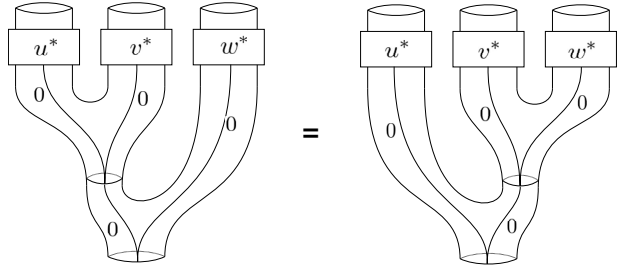}
\end{center}
\caption{The map $\tilde{m}$ is associative, that is, $\tilde{m}(w^{\ast}\otimes \tilde{m}(v^{\ast}\otimes u^{\ast}))=\tilde{m}(\tilde{m}(w^{\ast}\otimes v^{\ast})\otimes u^{\ast})$ for any $u^{\ast},v^{\ast},w^{\ast}\in\tilde{L}$.  }
\label{dualpicture}
\end{figure}
\begin{figure}[!h]
\begin{center}
\includegraphics[scale=0.51]{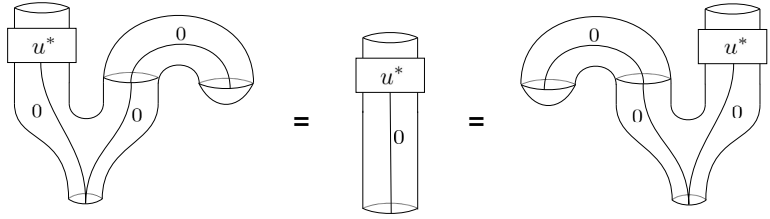}
\end{center}
\caption{The element $\tilde{1}$ is the unit, that is, $\tilde{m}( \tilde{1}\otimes u^{\ast})=u^{\ast}=\tilde{m}(u^{\ast}\otimes \tilde{1})$ for any $u^{\ast}\in\tilde{L}$.  }
\label{dualpicture2}
\end{figure}
\begin{figure}[!h]
\begin{center}
\includegraphics[scale=0.5]{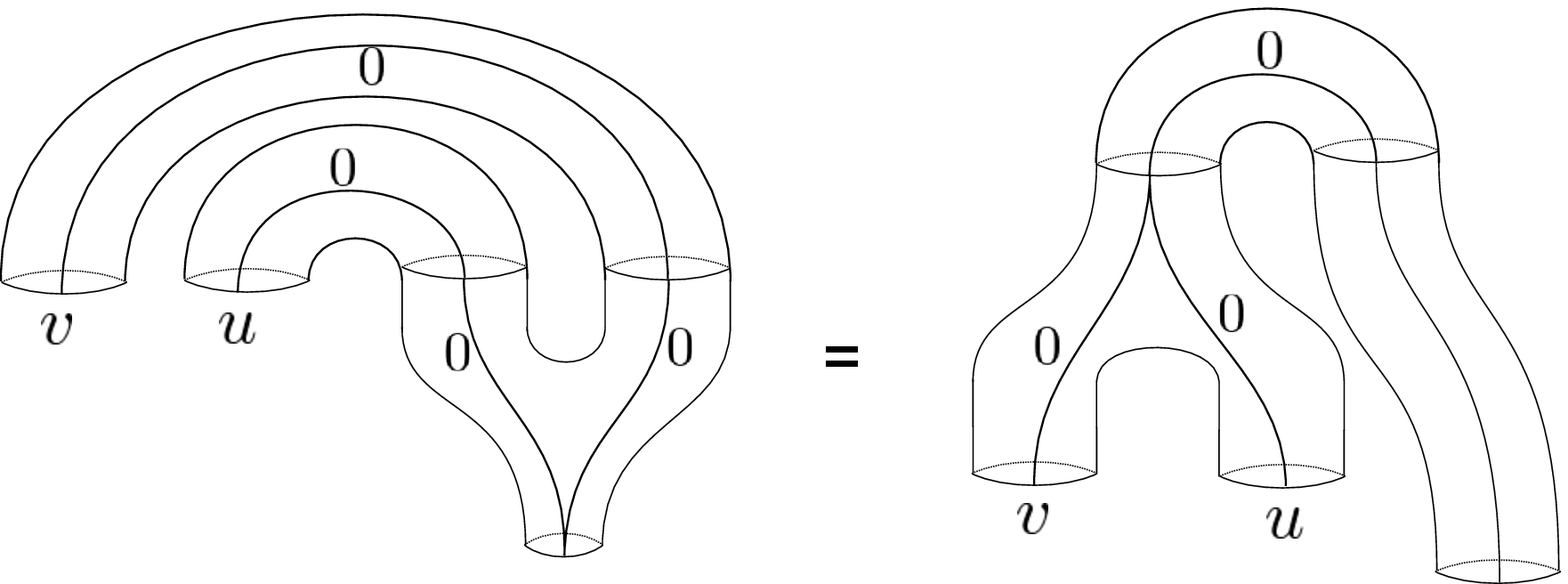}
\end{center}
\caption{The equation $(\ref{m})$ in the proof of Lemma~$\ref{duallem}$.  }
\label{dualpicture3}
\end{figure}
\begin{figure}[!h]
\begin{center}
\includegraphics[scale=0.55]{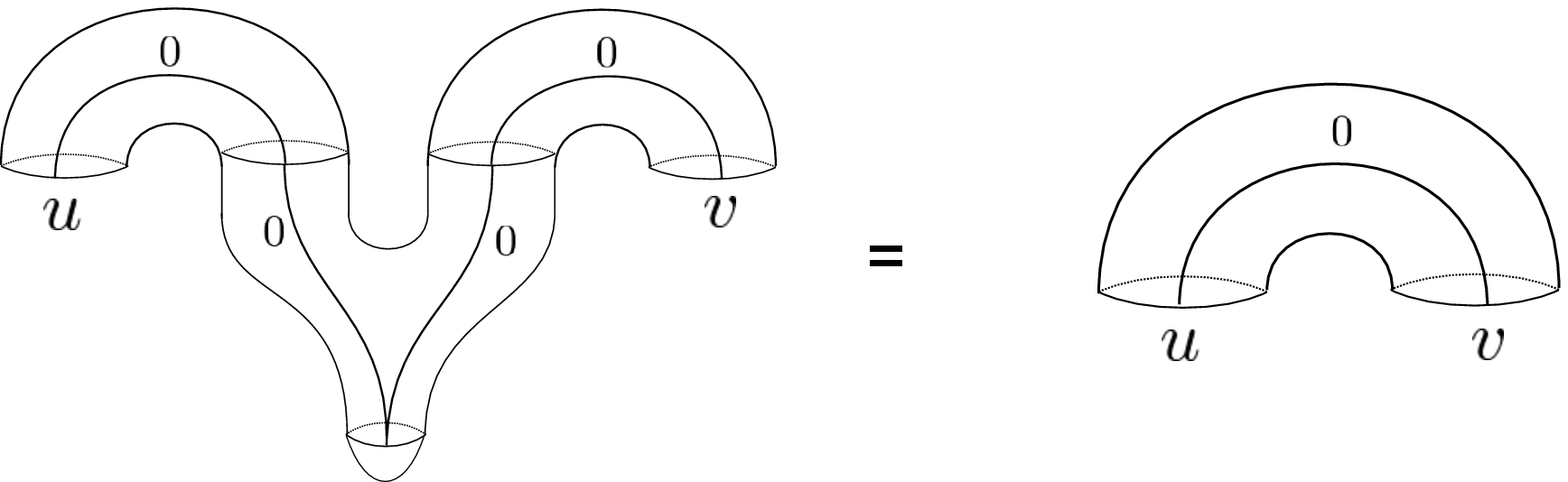}
\end{center}
\caption{The equation $(\ref{eta})$ in the proof of Lemma~$\ref{duallem}$.  }
\label{dualpicture4}
\end{figure}
\end{proof}
%
The isomorphism given in Lemma~$\ref{duallem}$ induces an isomorphism between $C_{(A,\tau)}^{i}(F,D)$ and the dual complex of $C_{(A,\tau)}^{i}(F,D^{!})$. 
\begin{proof}[{\bf Proof of Theorem~$\ref{dual}$}]
From the definition, we have 
\begin{align*}
C_{(A,\tau)}^{i}(F,D)=\bigoplus_{|\varepsilon|=i+n_{-}(D)}A(D_{\varepsilon}). 
\end{align*}
For $\varepsilon\in\{0,1\}^{n}$, we define $\tilde{\varepsilon}\in\{0,1\}^{n}$ by 
\begin{align*}
\tilde{\varepsilon_{i}}=
\begin{cases}
0 & \text{if\ } \varepsilon_{i}=1,\\
1 & \text{if\ } \varepsilon_{i}=0, \\
\end{cases}
\end{align*}
where $n$ is the number of crossings of $D$. 
Then we obtain $A(D^{!}_{\varepsilon})=A(D_{\tilde{\varepsilon}})$. 
From Lemma~$\ref{duallem}$, the following diagram is commutative. 
\begin{equation*}
\begin{CD}
A(D_{\varepsilon})  @>d_{(A,\tau)}^{i}(F, D)>> A(D_{\varepsilon'}) \\
@|  @| \\
A(D^{!}_{\tilde{\varepsilon}})   @.        A(D^{!}_{\tilde{\varepsilon'}}) \\
@V\cong Vv\mapsto \eta(v\otimes \cdot)V    @V\cong Vv\mapsto \eta(v\otimes \cdot)V  \\
A(D^{!}_{\tilde{\varepsilon}})^{\ast}  @>(d_{(A,\tau)}^{i}(F, D^{!}))^{\ast}>>  A(D^{!}_{\tilde{\varepsilon'}})^{\ast} 
\end{CD}
\end{equation*}
This diagram induces an isomorphisms 
\begin{align*}
C_{(A,\tau)}^{i}(F,D)\cong \{C_{(A,\tau)}^{i}(F,D^{!})\}^{\ast}. 
\end{align*}
\end{proof}
\begin{cor}
Let $(F,D)$ be an oriented link diagram on an oriented compact surface $F$ and $(F,D^{!})$ the link diagram obtained from $D$ by changing all crossings of $D$. 
For $\pi=H_{1}(F; \mathbf{F_{2}})$ and for any HQFT $(A, \tau)$ with target $K(\pi, 1)$ which preserves the S, T and 4-Tu relations given in Figure~$\ref{relation}$, we obtain 
\begin{align*}
H^{i}_{(A,\tau)}(F,D)\cong H^{-i}_{(A,\tau)}(F,D^{!}). 
\end{align*}
\end{cor}
\subsection{Examples of unoriented HQFTs over $\mathbf{F_{2}}$}\label{link_homology1}
In this subsection, we construct an example of unoriented HQFTs over $\mathbf{F_{2}}$. 
\par
Let $\pi$ be an $\mathbf{F_{2}}$-vector space. 
Define $L_{0}$ to be the $\mathbf{F_{2}}$-vector space with basis $1$ and $x$, and $L_{\alpha}$ to be the $\mathbf{F_{2}}$-vector space with basis $y_{\alpha}$ and $z_{\alpha}$ for $0\neq\alpha\in\pi$. 
Set $L=\bigoplus_{\alpha\in\pi}L_{\alpha}$. 
We give $L$ an extended crossed $\pi$-algebra structure as follows: 
\begin{itemize}
\item $\varphi_{\alpha}:=\operatorname{id}\colon L\rightarrow L$ for any $\alpha\in\pi$, 
\item $\Phi:=\operatorname{id}\colon L\rightarrow L$, 
\item $\theta_{\alpha}:=0$ for any $\alpha \in\pi$, 
\item for distinct elements $\alpha, \beta \in\pi\setminus \{0\}$, the multiplication $m\colon L\otimes L\rightarrow L$ is given as follows: 
\begin{align*}
m(v\otimes 1)=m(1\otimes v)=&v, \\
m(x\otimes x)=&0, \\
m(y_{\alpha}\otimes x)=m(x\otimes y_{\alpha})=&0, \\
m(z_{\alpha}\otimes x)=m(x\otimes z_{\alpha})=&0, \\
m(y_{\alpha}\otimes z_{\alpha})=m(z_{\alpha}\otimes y_{\alpha})=&x, \\
m(y_{\alpha}\otimes y_{\alpha})=m(z_{\alpha}\otimes z_{\alpha})=&0, \\
m(y_{\alpha}\otimes z_{\beta})=m(z_{\alpha}\otimes y_{\beta})=&y_{\alpha+\beta}+z_{\alpha+\beta}, \\
m(y_{\alpha}\otimes y_{\beta})=&y_{\alpha+\beta}+z_{\alpha+\beta}, \\
m(z_{\alpha}\otimes z_{\beta})=&y_{\alpha+\beta}+z_{\alpha+\beta}, 
\end{align*}
\item the inner product $\eta\colon L_{\alpha}\otimes L_{\alpha}\rightarrow \mathbf{F_{2}}$ is given as follows: 
\begin{align*}
\eta (1\otimes 1)=&0, \\
\eta (x\otimes 1)=\eta (1\otimes x)=&1, \\
\eta (x\otimes x)=&0, \\
\eta (y_{\alpha}\otimes y_{\alpha})=\eta (z_{\alpha}\otimes z_{\alpha})=&0,\\ 
\eta (y_{\alpha}\otimes z_{\alpha})=\eta (z_{\alpha}\otimes y_{\alpha})=&1 . \\
\end{align*}
\end{itemize}
From Theorem~$\ref{thm1}$ and the following proposition, we obtain an unoriented HQFT with target $K(\pi,1)$. 
\begin{prop}\label{prophqft}
Let $\pi$ be an $\mathbf{F_{2}}$-vector space. 
Then the algebra $(L=\bigoplus_{\alpha\in\pi}L_{\alpha}, \eta, \varphi, \Phi, \{\theta_{\alpha}\}_{\alpha\in\pi})$ constructed as above is an extended crossed $\pi$-algebra over $\mathbf{F_{2}}$. 
From Theorem~$\ref{thm1}$, we also obtain the unoriented HQFT $(A,\tau)$ with target $K(\pi,1)$ corresponding to this extended crossed $\pi$-algebra. 
\end{prop}
\begin{proof}
We can directly check that this algebra satisfies the axioms of extended crossed $\pi$-algebras. 
\end{proof}
\begin{rem}
The comultiplication $\Delta \colon L\rightarrow L\otimes L$ and the counit $\varepsilon \colon L_{0}\rightarrow \mathbf{F_{2}}$ of $L$ are given as follows: 
\begin{itemize}
\item $\Delta_{0,0} \colon L_{0}\rightarrow L_{0}\otimes L_{0}$ is 
\begin{align*}
\Delta_{0,0} (1)=&1\otimes x+x\otimes 1, \\
\Delta_{0,0} (x)=&x\otimes x, 
\end{align*}
\item $\Delta_{\alpha,\alpha} \colon L_{0}\rightarrow L_{\alpha}\otimes L_{\alpha}$ is 
\begin{align*}
\Delta_{\alpha,\alpha} (1)=&y_{\alpha} \otimes z_{\alpha} +z_{\alpha}\otimes y_{\alpha}, \\
\Delta_{\alpha,\alpha} (x)=&0, 
\end{align*}
\item $\Delta_{0,\alpha} \colon L_{\alpha}\rightarrow L_{0}\otimes L_{\alpha}$ is 
\begin{align*}
\Delta_{0,\alpha} (y_{\alpha})=&x\otimes y_{\alpha}, \\
\Delta_{0,\alpha} (z_{\alpha})=&x\otimes z_{\alpha}, 
\end{align*}
\item $\Delta_{\alpha,0} \colon L_{\alpha}\rightarrow L_{\alpha}\otimes L_{0}$ is $P\circ \Delta_{0,\alpha}$, where $P$ is the permutation, 
\item $\Delta_{\alpha,\beta} \colon L_{\alpha+\beta}\rightarrow L_{\alpha}\otimes L_{\beta}$ is 
\begin{align*}
\Delta_{\alpha,\beta} (y_{\alpha+\beta})=&y_{\alpha}\otimes y_{\beta}+y_{\alpha}\otimes z_{\beta}+z_{\alpha}\otimes y_{\beta}+z_{\alpha}\otimes z_{\beta}, \\
\Delta_{\alpha,\beta} (z_{\alpha+\beta})=&y_{\alpha}\otimes y_{\beta}+y_{\alpha}\otimes z_{\beta}+z_{\alpha}\otimes y_{\beta}+z_{\alpha}\otimes z_{\beta}, 
\end{align*}
\item $\varepsilon \colon L_{0}\rightarrow \mathbf{F_{2}}$ is 
\begin{align*}
\varepsilon (1)=&0, \\
\varepsilon (x)=&1. 
\end{align*}
\end{itemize}
\end{rem}
\begin{prop}\label{4-tu_prop}
Let $\pi$ be an $\mathbf{F_{2}}$-vector space and $(A,\tau)$ the HQFT established in Proposition~$\ref{prophqft}$. 
Suppose that $X$-cobordisms $W_{1}$, $W_{2}$, $W_{3}$ and $W_{4}$ are related as in the $4$-Tu relation depicted in Figure~$\ref{relation}$. 
Then we have $\tau(W_{1})+\tau(W_{2})=\tau(W_{3})+\tau(W_{4})$. 
\end{prop}
\begin{proof}
We only need to check the case depicted in Figure~$\ref{4-tu}$. 
In this case, we can compute 
\begin{align*}
\tau(W_{1})+\tau(W_{2})&=1\otimes 1\otimes(1\otimes x+x\otimes 1)+(1\otimes x+x\otimes 1)\otimes1\otimes 1, \\
\tau(W_{3})+\tau(W_{4})&=1\otimes 1\otimes x\otimes 1+x\otimes 1\otimes 1\otimes 1+1\otimes 1\otimes1\otimes x+1\otimes x\otimes1\otimes 1.  
\end{align*}
Hence we have $\tau(W_{1})+\tau(W_{2})=\tau(W_{3})+\tau(W_{4})$. 
\end{proof}
\begin{figure}[!h]
\begin{center}
\includegraphics[scale=0.6]{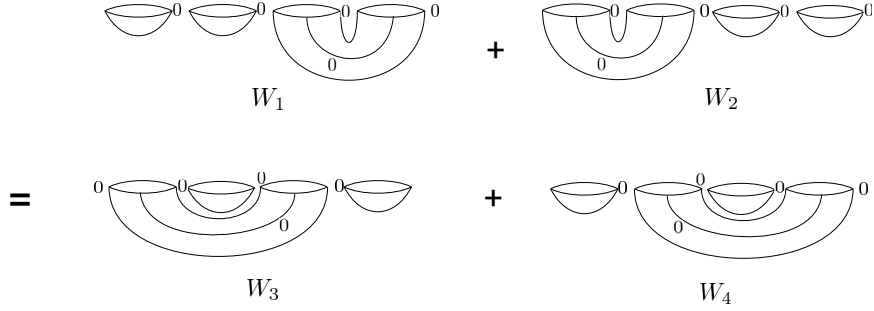}
\end{center}
\caption{$4$-Tube relation. }
\label{4-tu}
\end{figure}
\subsection{A non-vanishing condition}\label{section_nonvanish}
Let $(F,D)$ be an oriented link diagram $D$ on an oriented compact surface $F$ with $n_{-}(D)$ negative crossings and let $(A,\tau)$ be the HQFT constructed in Proposition~$\ref{prophqft}$. 
In this subsection, we find a condition where $H^{-n_{-}(D)}_{(A,\tau)}(F,D)\neq 0$ (Theorem~$\ref{non_vanish}$). 
Consider the complex: 
\begin{equation*}
\begin{CD}
0  @>>>C_{(A,\tau)}^{-n_{-}(D)}(F,D)@>d_{(A,\tau)}^{-n_{-}(D)}>> C_{(A,\tau)}^{-n_{-}(D)+1}(F,D)@>>>\cdots .\\
\end{CD}
\end{equation*}
By the definition, we have $H^{-n_{-}(D)}_{(A,\tau)}(F,D)=\ker d_{(A,\tau)}^{-n_{-}(D)}$. 
Suppose that $\ker d_{(A,\tau)}^{-n_{-}(D)}\neq 0$. 
Since $\Delta_{0,0}$, $\Delta_{0,\alpha}$ and $\Delta_{\alpha,0}$ $(\alpha \neq 0)$ are injective, the cobordisms defining ``$d_{(A,\tau)}^{-n_{-}(D)}$'' have no cbordism depicted in Figure~$\ref{injective_split}$. 
\begin{figure}[!h]
\begin{center}
\includegraphics[scale=0.5]{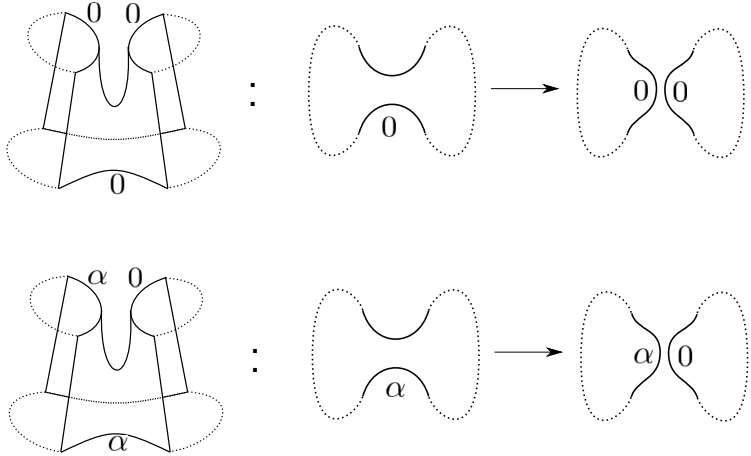}
\end{center}
\caption{The labels are the elements in $H_{1}(F;\mathbf{F_{2}})$ represented by the corresponding circles.  }
\label{injective_split}
\end{figure}
Hence the following condition (Definition~$\ref{condition}$) is necessary condition to be $\ker d_{(A,\tau)}^{-n_{-}(D)}\neq 0$. 
\begin{defn}\label{condition}
Let $D$ be an oriented link diagram on an oriented compact surface $F$ with $n$ crossings and let $D_{\mathbf{0}}$ be the smoothing where all crossings of $D$ are $0$-smoothed. 
Then $D$ is {\it weak plus-adequate} if for all $\varepsilon\in\{0,1\}^{n}$ with $|\varepsilon|=1$, the smoothing $D_{\varepsilon}$ is obtained from $D_{\mathbf{0}}$ by one of the following: 
\begin{itemize}
\item two circles of $D_{\mathbf{0}}$ merge into one circle of $D_{\varepsilon}$, 
\item a circle of $D_{\mathbf{0}}$ which is $0$-homologue splits into two circles of $D_{\varepsilon}$ which are not $0$-homologue (see Figure~$\ref{split1}$), 
\item a circle of $D_{\mathbf{0}}$ which is not $0$-homologue splits into two circles of $D_{\varepsilon}$ which are not $0$-homologue (see Figure~$\ref{split2}$), 
\item a circle of $D_{\mathbf{0}}$ changes into a circle of $D_{\varepsilon}$ (that is, the number of circles of $D_{\varepsilon}$ equals that of $D_{\mathbf{0}}$).  
\end{itemize}
\begin{figure}[!h]
\begin{center}
\includegraphics[scale=0.45]{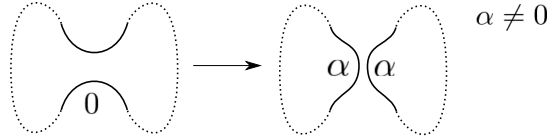}
\end{center}
\caption{A circle which is $0$-homologue splits into two circles which are not $0$-homologue.  }
\label{split1}
\end{figure}
\begin{figure}[!h]
\begin{center}
\includegraphics[scale=0.45]{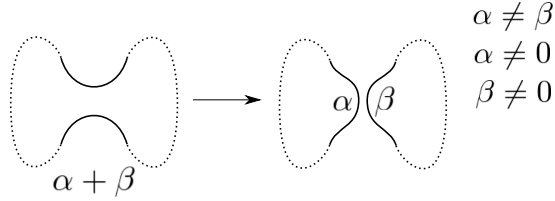}
\end{center}
\caption{A circle which is not $0$-homologue splits into two circles which are not $0$-homologue.  }
\label{split2}
\end{figure}
\end{defn}
\begin{rem}
A diagram is {\it plus-adequate} if, for  all $\varepsilon\in\{0,1\}^{n}$ with $|\varepsilon|=1$, the smoothing $D_{\varepsilon}$ is obtained from $D_{\mathbf{0}}$ by merging two circles into one circle.
Hence a plus-adequate diagram is weak plus-adequate. 
\end{rem}
Theorem~$\ref{non_vanish}$ ensures that the weak plus-adequacy is a sufficient and necessary condition to be $\ker d_{(A,\tau)}^{-n_{-}(D)}\neq 0$. 
\begin{proof}[{\bf Proof of Theorem~$\ref{non_vanish}$}]
If $(F,D)$ is not weak plus-adequate, from above discussion, $d_{(A,\tau)}^{-n_{-}(D)}$ is injective and $H^{-n_{-}(D)}_{(A,\tau)}(F,D)= 0$.
\par
If $(F,D)$ is weak plus-adequate, take an element $k$ in $A(D_{\mathbf{0}})$ as follows: 
Associate $x$ to each circle of $D_{\mathbf{0}}$ which is $0$-homologue and 
associate $y_{\alpha}+z_{\alpha}$ to each of the other circles of $D_{\mathbf{0}}$, where $\alpha$ is the corresponding element in $H_{1}(F;\mathbf{F_{2}})$ to the circle. 
Then $k$ is defined by the tensor product for these elements (see, for example, Figure~$\ref{non-zero-element}$). 
By below computation, $k$ is a non-zero element in $\ker d_{(A,\tau)}^{-n_{-}(D)}$. 
Hence we have $H^{-n_{-}(D)}_{(A,\tau)}(F,D)\neq 0$. 
\begin{align*}
x\cdot x &=0, \\
x\cdot (y_{\alpha}+z_{\alpha})&=0, \\
(y_{\alpha}+z_{\alpha})\cdot (y_{\alpha}+z_{\alpha})&=2x=0, \\
(y_{\alpha}+z_{\alpha})\cdot (y_{\beta}+z_{\beta})&=4(y_{\alpha+\beta}+z_{\alpha+\beta})=0, \\
\Delta_{\alpha,\alpha}(x)&=0, \\
\Delta_{\alpha,\beta}(y_{\alpha +\beta}+z_{\alpha +\beta})&=2(y_{\alpha}\otimes y_{\beta}+y_{\alpha}\otimes z_{\beta}+z_{\alpha}\otimes y_{\beta}+z_{\alpha}\otimes z_{\beta})=0, \\
\theta_{0} \cdot x&=\theta_{\alpha}\cdot (y_{\alpha}+z_{\alpha})=0. 
\end{align*}
\end{proof}
\begin{figure}[!h]
\begin{center}
\includegraphics[scale=0.65]{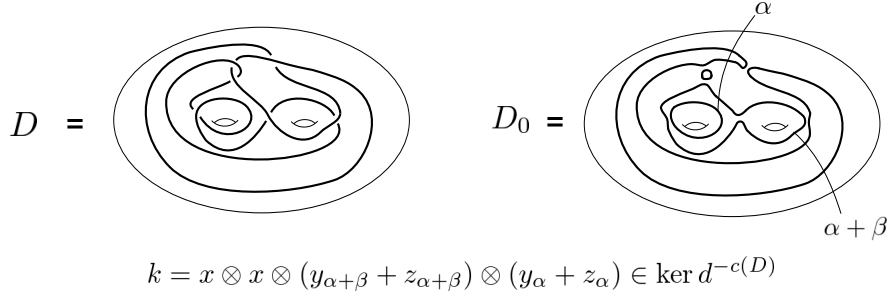}
\end{center}
\caption{An example of ``$k$". Let $D$ be the diagram on the oriented closed surface with genus $2$ depicted as above. Then $D_{\mathbf{0}}$ has two circles which are $0$-homologue and two circles which are not  $0$-homologue. 
Suppose that the homology classes of the circles are $\alpha$ and $\alpha+\beta$, respectively. 
Then, the element $k=x\otimes x\otimes(y_{\alpha+\beta}+z_{\alpha+\beta})\otimes(y_{\alpha}+z_{\alpha})$ is a non-zero element in $\ker d_{(A,\tau)}^{-c(D)}$. }
\label{non-zero-element}
\end{figure}
\begin{cor}\label{negative}
Let $(F,D)$ be an oriented link diagram $D$ on an oriented compact surface $F$ and let $(A, \tau)$ be the unoriented HQFT constructed in Proposition~$\ref{prophqft}$. 
Let $c_{-}(F, D)$ be the minimum over all the numbers of negative crossings of oriented link diagrams which are strongly equivalent to $(F, D)$. 
Then we obtain $c_{-}(F, D)\geq -i_{\min}(F,D):=-\min \{i\in\mathbf{Z}\mid H^{i}_{(A,\tau)}(F,D)\neq 0\}$. 
The equality holds if and only if $(F, D)$ is weak plus-adequate. 
\end{cor}
\begin{proof}
It follows from the definition of $H^{i}_{(A,\tau)}(F,D)$ and Theorem~$\ref{non_vanish}$. 
\end{proof}
\begin{rem}
Analogously, we define a weak minus-adequate diagram. 
Let $D$ be an oriented link diagram on an oriented compact surface $F$ with $n$ crossings and let $D_{\mathbf{1}}$ be the smoothing where all crossings of $D$ are $1$-smoothed. 
Then $D$ is {\it weak minus-adequate} if for all $\varepsilon\in\{0,1\}^{n}$ with $|\varepsilon|=n-1$, the smoothing $D_{\mathbf{1}}$ is obtained from $D_{\varepsilon}$ by one of the following: 
\begin{itemize}
\item a circles of $D_{\varepsilon}$ split into two circle of $D_{\mathbf{1}}$, 
\item two circles of $D_{\varepsilon}$ which are not $0$-homologue merge into a circle of $D_{\mathbf{1}}$ which is $0$-homologue, 
\item two circles of $D_{\varepsilon}$ which are not $0$-homologue merge into a circle of $D_{\mathbf{1}}$ which is not $0$-homologue, 
\item a circle of $D_{\varepsilon}$ changes into a circle of $D_{\mathbf{1}}$.  
\end{itemize}
Under the setting in Corollary~$\ref{negative}$, we also have 
$c_{+}(F, D)\geq i_{\max}(F,D):=\max \{i\in\mathbf{Z}\mid H^{i}_{(A,\tau)}(F,D)\neq 0\}$. 
From the duality relation (Theorem~$\ref{dual}$) and Corollary~$\ref{negative}$, we obtain the equality if and only if $(F, D)$ is weak minus-adequate. 
\end{rem}
\begin{cor}
Let $(F,D)$ be an oriented link diagram $D$ on an oriented compact surface $F$. 
Denote by $c(F, D)$ the minimal number of crossings of oriented link diagrams which are strongly equivalent to $(F, D)$. 
If $(F, D)$ is weak plus-adequate and weak minus-adequate, then the diagram $D$ is minimal, that is, it has $c(F, D)$ crossings. 
\end{cor}
\section{Examples of computations}\label{section_example}
\subsection{How to compute the homology}
In this subsection, we explain how to compute our homology defined in Section~$\ref{link_homology2}$. 
\par
Let $(F, D)$ be an oriented link diagram. 
Set $\pi=H_{1}(F; \mathbf{F}_{2})$. 
Let $(L=\bigoplus_{\alpha\in\pi}L_{\alpha}, \eta, \varphi, \{\theta_{\alpha}\}_{\alpha\in\pi}, \Phi)$ be the extended crossed $\pi$-algebra corresponding to an unoriented HQFT $(A, \tau)$ over $\mathbf{F}_{2}$ (see Definition~$\ref{extcross}$ and Theorem~$\ref{thm1}$). 
\par
Fistly, fix an order of the crossings of $D$. 
Let $n$ be the number of the crossings of $D$. 
Then for any $\varepsilon\in\{0,1\}^{n}$, we obtain the smoothing $D_{\varepsilon}$, that is, the $i$-th crossing of $D$ is $\varepsilon_{i}$-smoothed (see Figure~$\ref{smoothing}$ and Section~$\ref{geometric_cpx}$). 
Note that $D_{\varepsilon}$ is a collection of disjoint circles on $F$. 
\par
Secondly, we assign $L_{\alpha}$ to each circle in $D_{\varepsilon}$ whose homology class is $\alpha\in\pi$. 
The module $A(D_{\varepsilon})$ is given as the tensor product among the modules assigned to the circles in $D_{\varepsilon}$. 
Namely, 
\begin{align*}
A(D_{\varepsilon})=\bigotimes_{\alpha\in\pi}L_{\alpha}^{\otimes k_{\varepsilon}^{\alpha}}, 
\end{align*}
where $k_{\varepsilon}^{\alpha}$ is the number of the circles in $D_{\varepsilon}$ whose homology class is $\alpha\in\pi=H_{1}(F; \mathbf{F}_{2})$. 
Then $C^{i}_{(A, \tau)}(F, D)$ is given by 
\begin{align*}
C^{i}_{(A, \tau)}(F, D)=\bigoplus_{|\varepsilon|=i+n_{-}(D)}A(D_{\varepsilon}), 
\end{align*}
where $n_{-}(D)$ is the number of the negative crossings of $D$. 
\par
Thirdly, we recall the definition of the differential map $d^{i}_{(A, \tau)}\colon C^{i}_{(A, \tau)}(F, D) \rightarrow C^{i+1}_{(A, \tau)}(F, D)$. 
Take elements $\varepsilon$, $\varepsilon'\in\{0,1\}^{n}$ such that $\varepsilon_{j}=0$ and $\varepsilon'_{j}=1$ for some $j$ and that $\varepsilon_{i}=\varepsilon'_{i}$ for any $i\neq j$. 
For such a pair $(\varepsilon, \varepsilon')$ we will define a map $d_{\varepsilon\to \varepsilon'}\colon A(D_{\varepsilon})\to A(D_{\varepsilon'})$. 
\par
In the case where two circles of $D_{\varepsilon}$ merge into one circle of $D_{\varepsilon'}$, the map $d_{\varepsilon\to \varepsilon'}$ is the identity on all factors except the tensor factors corresponding to the merged circles where it is the multiplication $m\colon L_{\alpha}\otimes L_{\beta} \to L_{\alpha+\beta}$, where $\alpha, \beta \in \pi$ are the homology classes of the merged circles of $D_{\varepsilon}$. 
\par
In the case where one circle of $D_{\varepsilon}$ split into two circles of $D_{\varepsilon'}$, the map $d_{\varepsilon\to \varepsilon'}$ is the identity on all factors except the tensor factor corresponding to the split circle where it is the comultiplication $\Delta_{\alpha, \beta}\colon L_{\alpha+\beta} \to L_{\alpha}\otimes L_{\beta}$, where $\alpha, \beta \in \pi$ are the homology classes of the split circles of $D_{\varepsilon'}$ (the definition of $\Delta_{\alpha, \beta}$ is given in Definition~$\ref{extcross}$). 
\par
In the case where one circle of $D_{\varepsilon}$ changes into another circle of $D_{\varepsilon'}$, the map $d_{\varepsilon\to \varepsilon'}$ multiplies the tensor factor corresponding to the circle of $D_{\varepsilon}$ by $\theta_{0}$. 
\par
If there exist distinct integers $i$ and $j$ such that $\varepsilon_{i}\neq \varepsilon'_{i}$ and that $\varepsilon_{i}\neq \varepsilon'_{i}$, define $d_{\varepsilon\to \varepsilon'}=0$. 
\par
Then, the differential map $d^{i}_{(A, \tau)}$ are given as 
\begin{align*}
d^{i}_{(A, \tau)}=\sum_{|\varepsilon'|=i+1}d_{\varepsilon\to \varepsilon'}. 
\end{align*}
Finally, our homology is given by 
\begin{align*}
H^{\ast}_{(A, \tau)}(F, D)=H(C^{\ast}_{(A, \tau )}(F, D), d^{i}). 
\end{align*}
In next subsection, we give some computational examples. 
\subsection{Some computations}
In this subsection, we compute our homologies of some link diagrams. 
Let $(A,\tau)$ be the HQFT constructed in Proposition~$\ref{prophqft}$. 
\begin{example}\label{computation1}
Let $(F,D)$ be the oriented link diagram depicted in Figure~$\ref{sample-1}$. 
Then we have 
\begin{align*}
H^{i}_{(A,\tau)}(F, D)=
\begin{cases}
(\mathbf{F_{2}})^{2} & \text{if\ }i=0,\\
0& \text{otherwise}. 
\end{cases}
\end{align*}
\end{example}
\begin{figure}[!h]
\begin{center}
\includegraphics[scale=0.2]{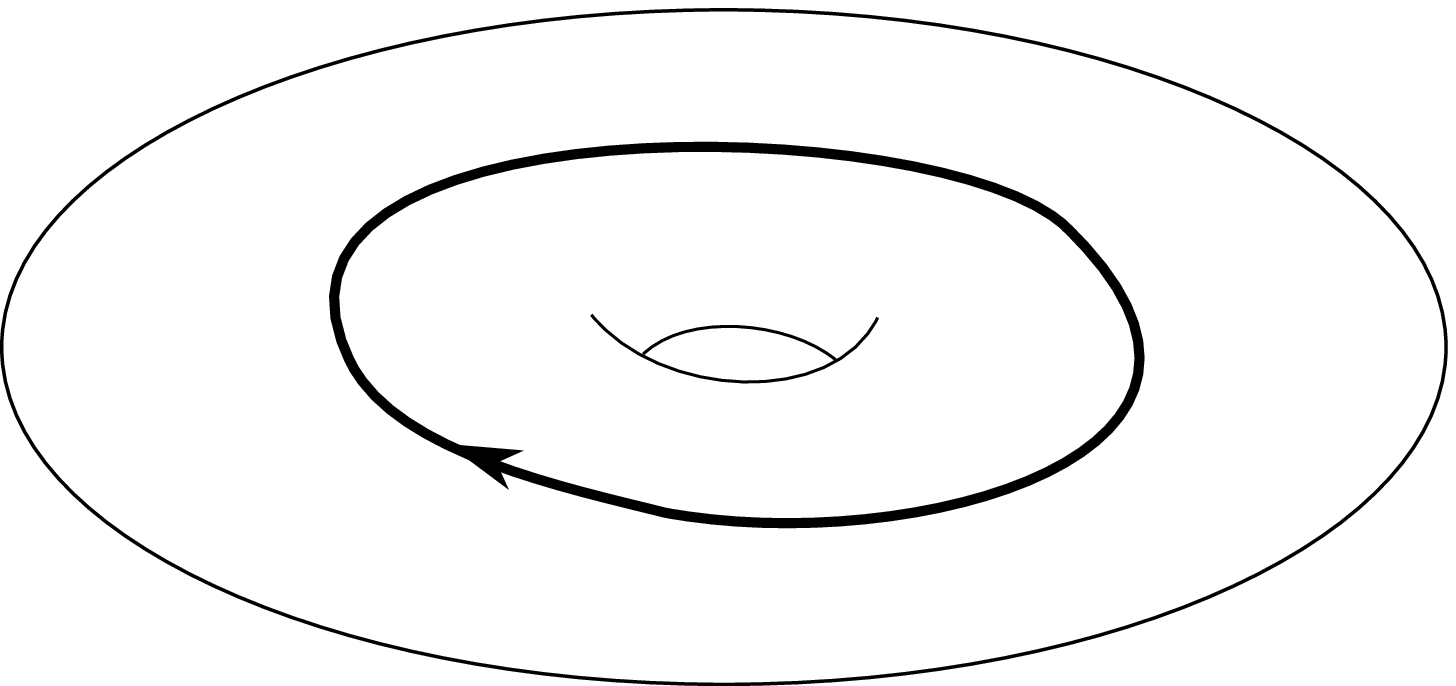}
\end{center}
\caption{Example~$\ref{computation1}$. }
\label{sample-1}
\end{figure}
%
\begin{example}\label{computation2}
Let $(F,D)$ be the oriented link diagram depicted in Figure~$\ref{sample-2}$. 
The complex $(C^{\ast}_{(A,\tau)}(F,D), d_{(A,\tau)}^{\ast})$ is as follows:  
\begin{align*}
C_{(A,\tau)}^{0}(F,D)&=L_{\alpha}\otimes L_{\alpha}, \\
C_{(A,\tau)}^{1}(F,D)&=L_{0}, \\
d_{(A,\tau)}^{0}&=m_{\alpha, \alpha}\colon L_{\alpha}\otimes L_{\alpha}\rightarrow L_{0} , 
\end{align*}
where $\alpha$ is the homology class of the longitude of $F$. 
We can compute 
\begin{align*}
\ker d_{(A,\tau)}^{0}&=\operatorname{Span}_{\mathbf{F_{2}}}\{y_{\alpha}\otimes y_{\alpha}, z_{\alpha}\otimes z_{\alpha}, y_{\alpha}\otimes z_{\alpha}+z_{\alpha}\otimes y_{\alpha}\}, \\
\operatorname{Im} d_{(A,\tau)}^{0} &=\operatorname{Span}_{\mathbf{F_{2}}}\{x\}. 
\end{align*}
Hence we have 
\begin{align*}
H^{i}_{(A,\tau)}(F, D)=
\begin{cases}
(\mathbf{F_{2}})^{3} & \text{if\ }i=0,\\
\mathbf{F_{2}} & \text{if\ }i=1,\\
0& \text{otherwise}. 
\end{cases}
\end{align*}
\end{example}
\begin{figure}[!h]
\begin{center}
\includegraphics[scale=0.3]{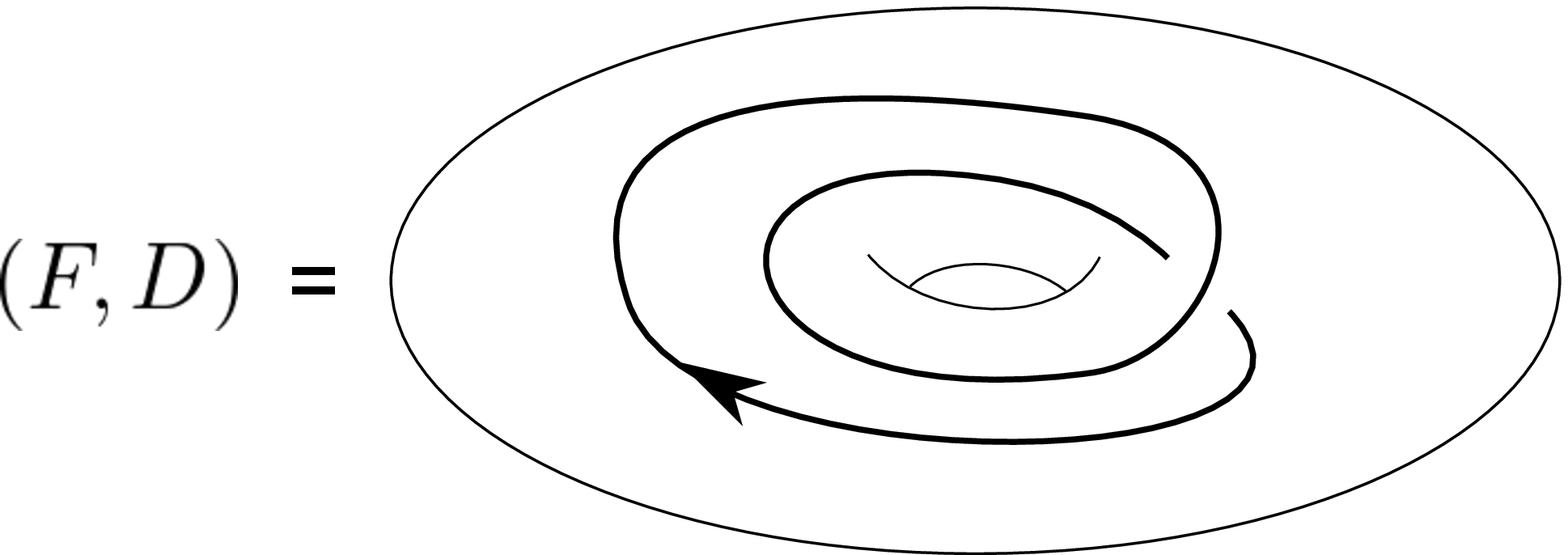}
\end{center}
\begin{center}
\vspace{10pt}
\includegraphics[scale=0.5]{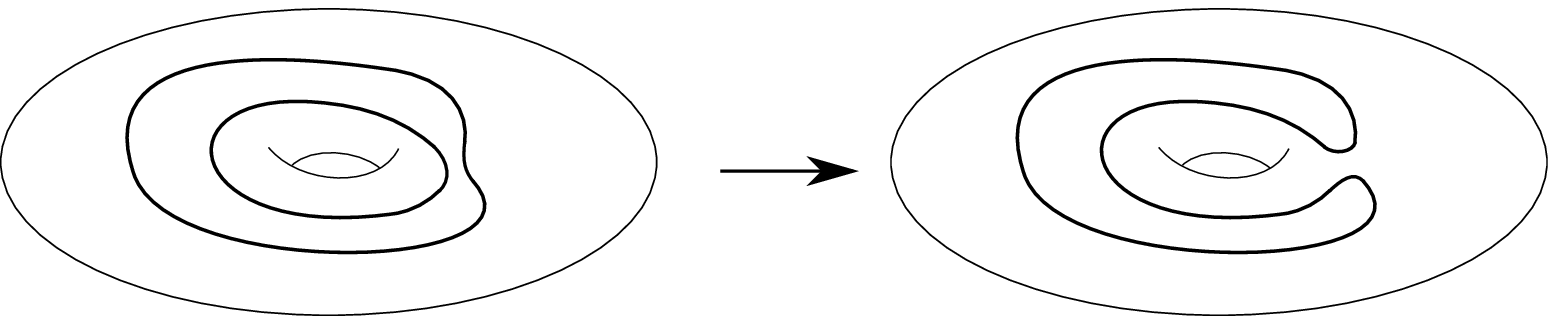}
\end{center}
\caption{Example~$\ref{computation2}$. }
\label{sample-2}
\end{figure}
\begin{example}\label{computation3}
Let $(F,D)$ be the oriented link diagram depicted in Figure~$\ref{sample-3}$. 
The complex $(C^{\ast}_{(A,\tau)}(F,D), d_{(A,\tau)}^{\ast})$ is as follows:  
\begin{align*}
C_{(A,\tau)}^{0}(F,D)&=L_{\alpha}\otimes L_{\beta}, \\
C_{(A,\tau)}^{1}(F,D)&=L_{\alpha+\beta}\oplus L_{\alpha+\beta}, \\
C_{(A,\tau)}^{2}(F,D)&=L_{\alpha+\beta}\otimes L_{0}, \\
d_{(A,\tau)}^{0}&=(m_{\alpha, \beta}, m_{\alpha, \beta})\colon L_{\alpha}\otimes L_{\beta}\rightarrow L_{\alpha+\beta}\oplus L_{\alpha+\beta} , \\
d_{(A,\tau)}^{1}&=\Delta_{\alpha, \beta}+ \Delta_{\alpha, \beta}\colon L_{\alpha+\beta}\oplus L_{\alpha+\beta}\rightarrow L_{\alpha+\beta}\otimes L_{0} , \\
\end{align*}
where $\alpha$ and $\beta$ are the homology classes of the corresponding circles. 
We can compute 
\begin{align*}
\ker d_{(A,\tau)}^{0}&=\operatorname{Span}_{\mathbf{F_{2}}}\{y_{\alpha} \otimes y_{\beta}+y_{\alpha} \otimes z_{\beta}, y_{\alpha} \otimes y_{\beta}+z_{\alpha} \otimes y_{\beta}, y_{\alpha} \otimes y_{\beta}+z_{\alpha} \otimes z_{\beta}\}, \\
\operatorname{Im} d_{(A,\tau)}^{0} &=\operatorname{Span}_{\mathbf{F_{2}}}\{ (y_{\alpha+\beta}+z_{\alpha+\beta}, y_{\alpha+\beta}+z_{\alpha+\beta})\}, \\ 
\ker d_{(A,\tau)}^{1}&=\operatorname{Span}_{\mathbf{F_{2}}}\{(y_{\alpha+\beta}, y_{\alpha+\beta}), (z_{\alpha+\beta}, z_{\alpha+\beta})\}, \\
\operatorname{Im} d_{(A,\tau)}^{1} &=\operatorname{Span}_{\mathbf{F_{2}}}\{y_{\alpha+\beta}\otimes x, z_{\alpha+\beta}\otimes x\}. 
\end{align*}
Hence we have 
\begin{align*}
H^{i}_{(A,\tau)}(F, D)=
\begin{cases}
(\mathbf{F_{2}})^{3} & \text{if\ }i=0,\\
\mathbf{F_{2}} & \text{if\ }i=1,\\
(\mathbf{F_{2}})^{2} & \text{if\ }i=2,\\
0& \text{otherwise}. 
\end{cases}
\end{align*}
\end{example}
\begin{figure}[!h]
\begin{center}
\includegraphics[scale=0.35]{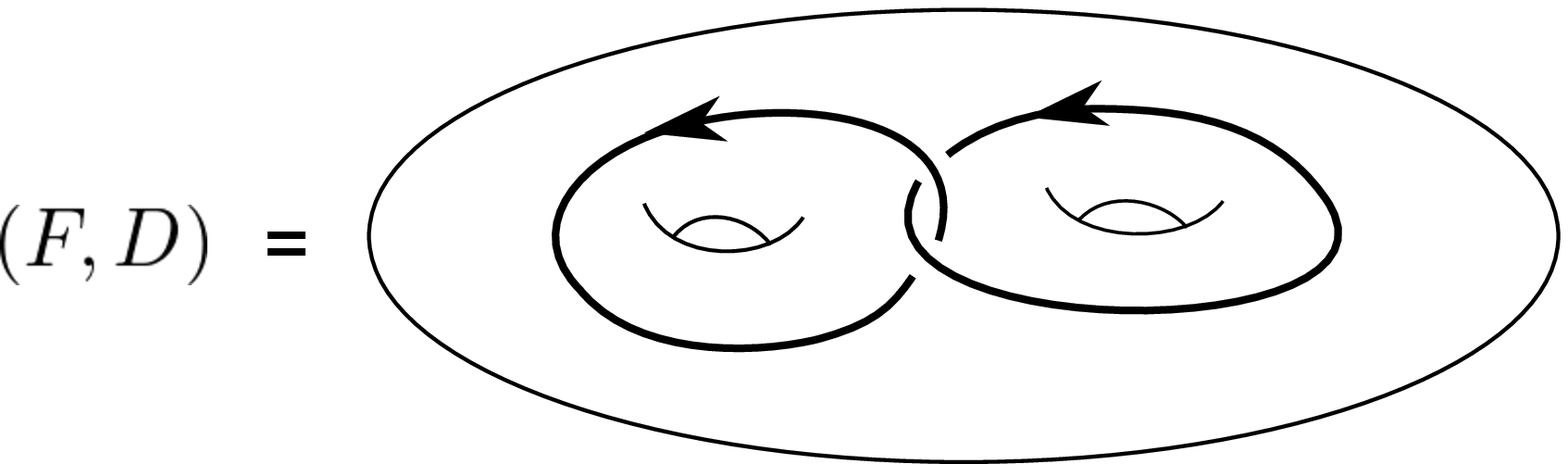}
\end{center}
\begin{center}
\vspace{10pt}
\includegraphics[scale=0.57]{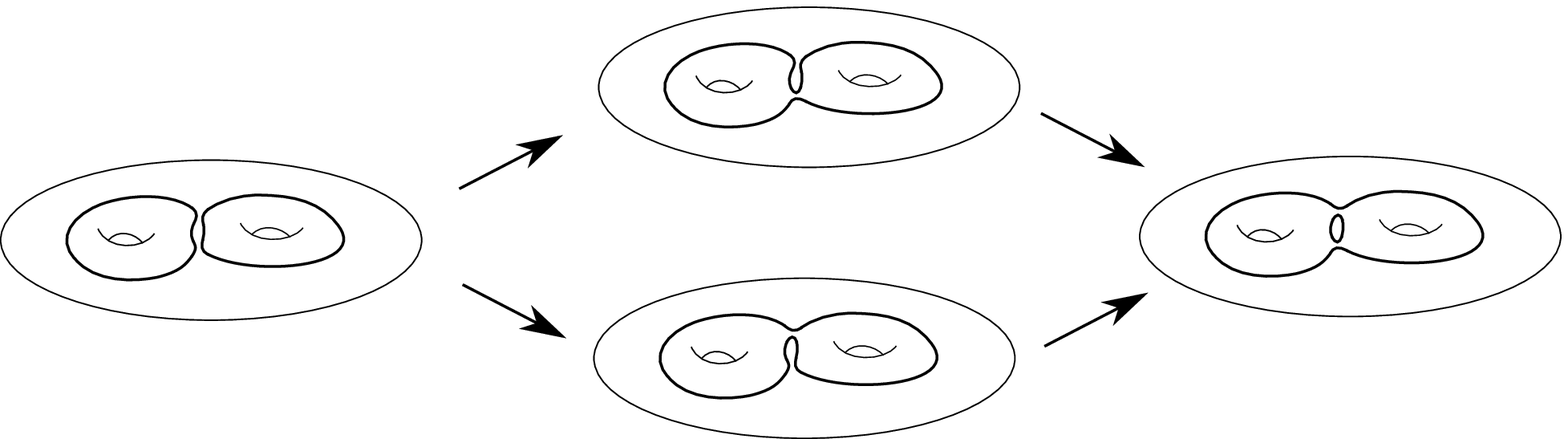}
\end{center}
\caption{Example~$\ref{computation3}$. }
\label{sample-3}
\end{figure}
\begin{example}\label{computation4}
Let $(F,D)$ be the oriented link diagram depicted in Figure~$\ref{sample-4}$. 
The complex $(C^{\ast}_{(A,\tau)}(F,D), d_{(A,\tau)}^{\ast})$ is as follows:  
\begin{align*}
C_{(A,\tau)}^{0}(F,D)&=L_{\alpha}, \\
C_{(A,\tau)}^{1}(F,D)&=L_{\alpha}\oplus L_{\alpha}, \\
C_{(A,\tau)}^{2}(F,D)&=L_{\alpha}\otimes L_{0}, \\
d_{(A,\tau)}^{0}&=(\theta_{\alpha}, \theta_{\alpha})=0\colon L_{\alpha}\rightarrow L_{\alpha}\oplus L_{\alpha} , \\
d_{(A,\tau)}^{1}&=\Delta_{\alpha, 0}+ \Delta_{\alpha, 0}\colon L_{\alpha}\oplus L_{\alpha}\rightarrow L_{\alpha}\otimes L_{0} , \\
\end{align*}
where $\alpha$ is the homology class of the corresponding circle. 
We can compute 
\begin{align*}
\ker d_{(A,\tau)}^{0}&=C_{(A,\tau)}^{0}(F,D)=L_{\alpha}, \\
\operatorname{Im} d_{(A,\tau)}^{0} &=0, \\ 
\ker d_{(A,\tau)}^{1}&=\operatorname{Span}_{\mathbf{F_{2}}}\{(y_{\alpha}, y_{\alpha}), (z_{\alpha}, z_{\alpha})\}, \\
\operatorname{Im} d_{(A,\tau)}^{1} &=\operatorname{Span}_{\mathbf{F_{2}}}\{y_{\alpha}\otimes x, z_{\alpha}\otimes x\}. 
\end{align*}
Hence we have 
\begin{align*}
H^{i}_{(A,\tau)}(F, D)=
\begin{cases}
\mathbf{F_{2}} & \text{if\ }i=0,\\
(\mathbf{F_{2}})^{2} & \text{if\ }i=1,\\
(\mathbf{F_{2}})^{2} & \text{if\ }i=2,\\
0& \text{otherwise}. 
\end{cases}
\end{align*}
\end{example}
\begin{figure}[!h]
\begin{center}
\includegraphics[scale=0.3]{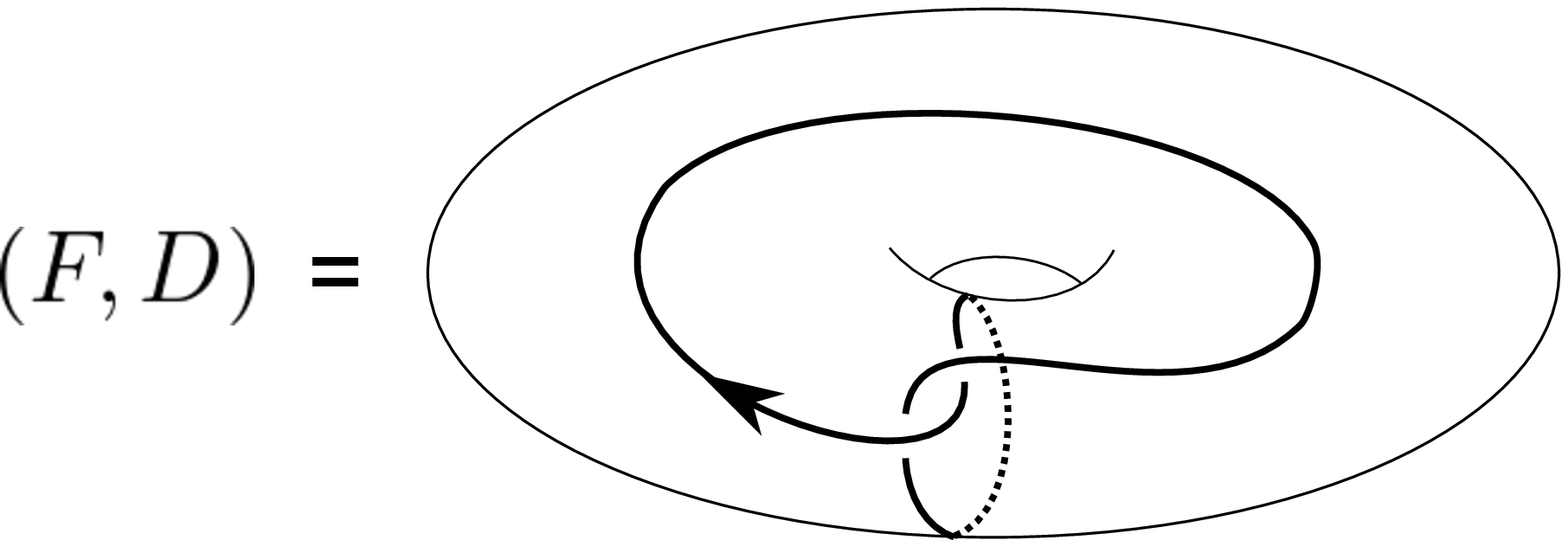}
\end{center}
\begin{center}
\vspace{10pt}
\includegraphics[scale=0.57]{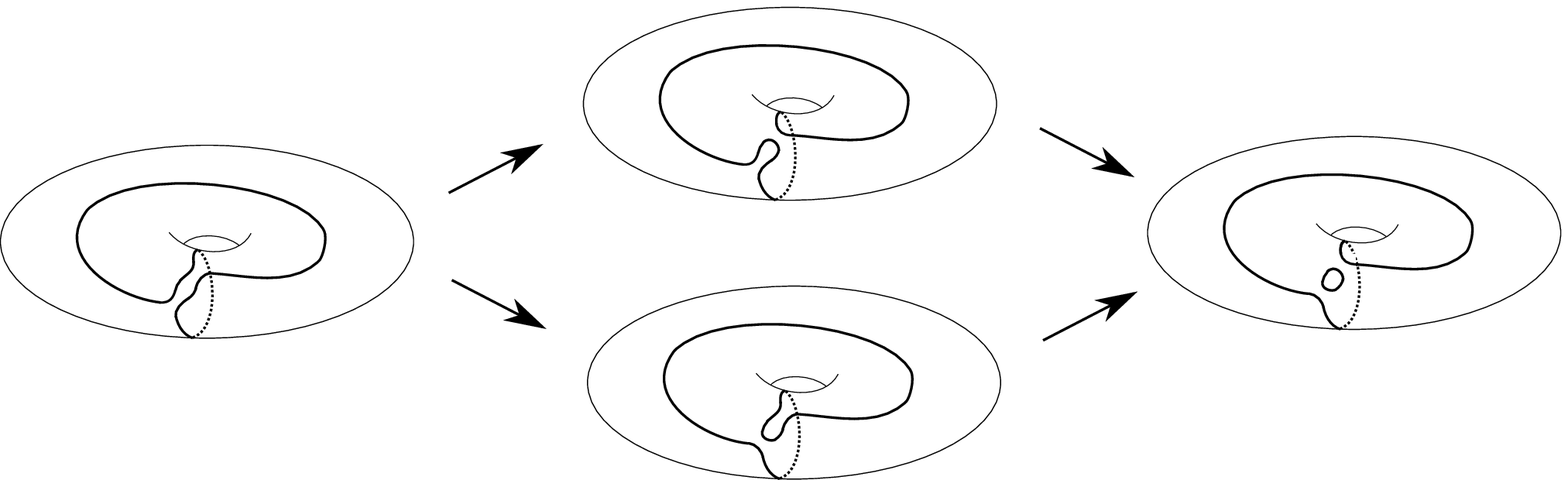}
\end{center}
\caption{Example~$\ref{computation4}$. }
\label{sample-4}
\end{figure}
\section{Remarks}\label{remark}
In this section, we make some remarks on our link homology theory. 
In the first remark, we make a remark on link diagrams on a non-orientable surface. 
In the second remark, we consider a group $G$ instead of $\pi=H_{1}(F;\mathbf{F_{2}})$. 
In the third, we give a new HQFT and show that the corresponding our link homology has ``q-grading". 
Finally, we construct a link homology which is an extension of Lee homology with coefficient $\mathbf{F_{2}}$ to strong equivalence classes. 
\begin{rem}\label{non-orientable}
Let $M$ be an oriented $I$-bundle over a compact surface $F$ (including non-orientable ones). 
Hence, $M$ is either $F\times I$ for an orientable $F$, or the twisted bundle of an non-orientable $F$. 
An oriented link in $M$ can be represented by a diagram on the surface. 
Two diagrams $(F, D)$ and $(F, D')$ represent the same link in $M$ if they are related by the following two operations: 
\begin{itemize} 
\item finite sequence of Reidemeister moves on the surfaces, 
\item homeomorphisms of the surfaces which are given by restricting orientation preserving bundle maps of $M$ to the base space $F$. 
\end{itemize}
In the case where $F$ is oriented, two link diagrams $(F, D)$ and $(F, D')$ represent the same link in $M$ if and only if $(F,D)$ is strongly equivalent to $(F, D')$. 
From the same discussion in Sections~$\ref{geometric_cpx}$ and $\ref{link_homology}$, the group $H^{i}_{(A, \tau)}(F,D)$ is an invariant of oriented links in $M$ ($F$ may be non-orientable) under ambient isotopy 
(compare this homology with \cite{link_homology_surface}). 
It also satisfies Theorems~$\ref{dual}$ and $\ref{non_vanish}$. 
\end{rem}
\begin{rem}\label{general_group}
Let $G$ be a group and $(F,D)$ a link diagram on an oriented compact surface. 
Let $q\colon H_{1}(F; \mathbf{F_{2}})\rightarrow G$ be a group homomorphism. 
Then, we can regard each smoothing $D_{\varepsilon}$ of $D$ as an unoriented closed $1$-dimensional $K(G, 1)$-manifold by assigning the element $q(\alpha)\in G$ to each circle $\Gamma$ of $D_{\varepsilon}$, where $\alpha=[\Gamma] \in H_{1}(F, \mathbf{F_{2}})$ is the homology class represented by $\Gamma$. 
Moreover, we obtain the corresponding geometric chain complex $(([[(F, D)]]_{(G, q)})^{i}, \Sigma^{i})$ (compare Section~$\ref{geometric_cpx}$). 
Obviously,  $(([[(F, D)]]_{(\{0\}, 0)})^{i}, \Sigma^{i})$ is Turaev and Turner's complex \cite[Proposition~$3.1$]{turner-turaev:2006} and $(([[(F, D)]]_{(H_{1}(F,\mathbf{F_{2}}), \operatorname{id})})^{i}, \Sigma^{i})$ is ours in Proposition~$\ref{mainprop}$. 
\end{rem}
\begin{rem}
Let $(L'=\bigoplus_{\alpha\in\pi}L'_{\alpha}, \eta', \varphi' , \Phi', \{\theta'_{\alpha}\}_{\alpha\in\pi})$ be as below. 
Then $L'$ is an extended crossed $\pi$-algebra and the corresponding HQFT $(A', \tau')$ preserves the S, T and $4$-Tu relations. 
From Theorem~$\ref{homology_hqft}$, we obtain a link homology $H_{(A', \tau')}^{\ast}$. 
\begin{itemize}
\item $\pi$ is an $\mathbf{F_{2}}$-vector space, 
\item $L'_{0}$ is the $\mathbf{F_{2}}$-vector space generated by $1$ and $x$, and $L'_{\alpha}$ is the $\mathbf{F_{2}}$-vector space generated by $y_{\alpha}$ and $z_{\alpha}$ for $0\neq\alpha\in\pi$, 
\item $\varphi'_{\alpha}:=\operatorname{id}\colon L'\rightarrow L'$ for any $\alpha\in\pi$, 
\item $\Phi:=\operatorname{id}\colon L'\rightarrow L'$, 
\item $\theta'_{\alpha}:=0$ for any $\alpha \in\pi$, 
\item for distinct elements $\alpha, \beta \in\pi\setminus \{0\}$, the multiplication $m'\colon L'\otimes L'\rightarrow L'$ is given as follows: 
\begin{align*}
m'(v\otimes 1)=m'(1\otimes v)=&v, \\
m'(x\otimes x)=&0, \\
m'(y_{\alpha}\otimes x)=m'(x\otimes y_{\alpha})=m'(z_{\alpha}\otimes x)=m'(x\otimes z_{\alpha})=&0, \\
m'(y_{\alpha}\otimes z_{\alpha})=m'(z_{\alpha}\otimes y_{\alpha})=&x, \\
m'(y_{\alpha}\otimes y_{\alpha})=m'(z_{\alpha}\otimes z_{\alpha})=&0, \\
m'(y_{\alpha}\otimes z_{\beta})=m'(z_{\alpha}\otimes y_{\beta})=&0, \\
m'(y_{\alpha}\otimes y_{\beta})=m'(z_{\alpha}\otimes z_{\beta})=&0, 
\end{align*}
\item the inner product $\eta'\colon L'_{\alpha}\otimes L'_{\alpha}\rightarrow \mathbf{F_{2}}$ is given as follows: 
\begin{align*}
\eta' (1\otimes 1)=\eta' (x\otimes x)=&0, \\
\eta' (x\otimes 1)=\eta'' (1\otimes x)=&1, \\
\eta' (y_{\alpha}\otimes y_{\alpha})=\eta' (z_{\alpha}\otimes z_{\alpha})=&0, \\
\eta' (y_{\alpha}\otimes z_{\alpha})=\eta' (z_{\alpha}\otimes y_{\alpha})=&1.  \\
\end{align*}
\end{itemize}
The comultiplication $\Delta' \colon L'\rightarrow L'\otimes L'$ and the counit $\varepsilon' \colon L'_{0}\rightarrow \mathbf{F_{2}}$ are given as follows: 
\begin{itemize}
\item $\Delta'_{0,0} \colon L'_{0}\rightarrow L'_{0}\otimes L'_{0}$ is 
\begin{align*}
\Delta'_{0,0} (1)=&1\otimes x+x\otimes 1, \\
\Delta'_{0,0} (x)=&x\otimes x, 
\end{align*}
\item $\Delta'_{\alpha,\alpha} \colon L'_{0}\rightarrow L'_{\alpha}\otimes L'_{\alpha}$ is 
\begin{align*}
\Delta'_{\alpha,\alpha} (1)=&y_{\alpha} \otimes z_{\alpha} +z_{\alpha}\otimes y_{\alpha}, \\
\Delta'_{\alpha,\alpha} (x)=&0, 
\end{align*}
\item $\Delta'_{0,\alpha} \colon L'_{\alpha}\rightarrow L'_{0}\otimes L'_{\alpha}$ is 
\begin{align*}
\Delta'_{0,\alpha} (y_{\alpha})=&x\otimes y_{\alpha}, \\
\Delta'_{0,\alpha} (z_{\alpha})=&x\otimes z_{\alpha}, 
\end{align*}
\item $\Delta'_{\alpha,0} \colon L'_{\alpha}\rightarrow L'_{\alpha}\otimes L'_{0}$ is $P\circ \Delta'_{0,\alpha}$, where $P$ is the permutation, 
\item $\Delta'_{\alpha,\beta} \colon L'_{\alpha+\beta}\rightarrow L'_{\alpha}\otimes L'_{\beta}$ is 
\begin{align*}
\Delta'_{\alpha,\beta} (y_{\alpha+\beta})=\Delta'_{\alpha,\beta} (z_{\alpha+\beta})=&0, 
\end{align*}
\item $\varepsilon' \colon L'_{0}\rightarrow \mathbf{F_{2}}$ is 
\begin{align*}
\varepsilon' (1)=&0, \\
\varepsilon' (x)=&1. 
\end{align*}
\end{itemize}
Moreover, this homology has ``q-grading" as follows: For an integer $t$, we define
\begin{align*}
\deg (1)&=1, \deg (x)=-1, \deg (y_{\alpha})=-\deg (z_{\alpha})=t, \\
H^{i, j}_{(A', \tau')}(F,D):&=\{v\in H_{(A', \tau')}^{i}(F,D)\mid \deg (v)+i+c_{+}(D)-c_{-}(D)=j\}\cup \{0\}, \\
H^{i}_{(A', \tau')}(F,D)&=\bigoplus _{j=-\infty}^{\infty}H^{i, j}_{(A', \tau')}(F,D), 
\end{align*}
where $c_{+}(D)$ and $c_{-}(D)$ are the numbers of the positive and negative crossings of $D$, respectively.  
This homology is an extension of the Khovanov homology with coefficient $\mathbf{F_{2}}$ to strong equivalence classes. 
\end{rem}
\begin{rem}
\par
Let $(L''=\bigoplus_{\alpha\in\pi}L''_{\alpha}, \eta'', \varphi'' , \Phi'', \{\theta''_{\alpha}\}_{\alpha\in\pi})$ be as below. 
Then $L''$ is an extended crossed $\pi$-algebra and the corresponding HQFT $(A'', \tau'')$ preserves the S, T and $4$-Tu relations. 
From Theorem~$\ref{homology_hqft}$, we obtain a link homology $H_{(A'', \tau'')}^{\ast}$. 
This homology is an extension of Lee homology with coefficient $\mathbf{F_{2}}$ to strong equivalence classes. 
\begin{itemize}
\item $\pi$ is an $\mathbf{F_{2}}$-vector space, 
\item $L''_{0}$ is the $\mathbf{F_{2}}$-vector space generated by $1$ and $x$, and $L''_{\alpha}$ is the $\mathbf{F_{2}}$-vector space generated by $y_{\alpha}$ and $z_{\alpha}$ for $0\neq\alpha\in\pi$, 
\item $\varphi''_{\alpha}:=\operatorname{id}\colon L''\rightarrow L''$ for any $\alpha\in\pi$, 
\item $\Phi'':=\operatorname{id}\colon L''\rightarrow L''$, 
\item $\theta''_{\alpha}:=0$ for any $\alpha \in\pi$, 
\item for distinct elements $\alpha, \beta \in\pi\setminus \{0\}$, the multiplication $m''\colon L''\otimes L''\rightarrow L''$ is given as follows: 
\begin{align*}
m''(v\otimes 1)=m''(1\otimes v)=&v, \\
m''(x\otimes x)=&1, \\
m''(y_{\alpha}\otimes x)=m''(x\otimes y_{\alpha})=&y_{\alpha}, \\
m''(z_{\alpha}\otimes x)=m''(x\otimes z_{\alpha})=&z_{\alpha}, \\
m''(y_{\alpha}\otimes z_{\alpha})=m''(z_{\alpha}\otimes y_{\alpha})=&x+1, \\
m''(y_{\alpha}\otimes y_{\alpha})=m''(z_{\alpha}\otimes z_{\alpha})=&0, \\
m''(y_{\alpha}\otimes z_{\beta})=m''(z_{\alpha}\otimes y_{\beta})=&0, \\
m''(y_{\alpha}\otimes y_{\beta})=m''(z_{\alpha}\otimes z_{\beta})=&0, 
\end{align*}
\item the inner product $\eta''\colon L''_{\alpha}\otimes L''_{\alpha}\rightarrow \mathbf{F_{2}}$ is given as follows: 
\begin{align*}
\eta'' (1\otimes 1)=&0, \\
\eta'' (x\otimes 1)=\eta'' (1\otimes x)=&1, \\
\eta'' (x\otimes x)=&0, \\
\eta'' (y_{\alpha}\otimes y_{\alpha})=\eta'' (z_{\alpha}\otimes z_{\alpha})=&0 ,\\
\eta'' (y_{\alpha}\otimes z_{\alpha})=\eta'' (z_{\alpha}\otimes y_{\alpha})=&1. \\
\end{align*}
\end{itemize}
The comultiplication $\Delta'' \colon L''\rightarrow L''\otimes L''$ and the counit $\varepsilon'' \colon L''_{0}\rightarrow \mathbf{F_{2}}$ are given as follows: 
\begin{itemize}
\item $\Delta''_{0,0} \colon L''_{0}\rightarrow L''_{0}\otimes L''_{0}$ is 
\begin{align*}
\Delta''_{0,0} (1)=&1\otimes x+x\otimes 1, \\
\Delta''_{0,0} (x)=&x\otimes x+1\otimes 1, 
\end{align*}
\item $\Delta''_{\alpha,\alpha} \colon L''_{0}\rightarrow L''_{\alpha}\otimes L''_{\alpha}$ is 
\begin{align*}
\Delta''_{\alpha,\alpha} (1)=&y_{\alpha} \otimes z_{\alpha} +z_{\alpha}\otimes y_{\alpha}, \\
\Delta''_{\alpha,\alpha} (x)=&y_{\alpha} \otimes z_{\alpha} +z_{\alpha}\otimes y_{\alpha}, 
\end{align*}
\item $\Delta''_{0,\alpha} \colon L''_{\alpha}\rightarrow L''_{0}\otimes L''_{\alpha}$ is 
\begin{align*}
\Delta''_{0,\alpha} (y_{\alpha})=&x\otimes y_{\alpha}, \\
\Delta''_{0,\alpha} (z_{\alpha})=&x\otimes z_{\alpha}, 
\end{align*}
\item $\Delta''_{\alpha,0} \colon L''_{\alpha}\rightarrow L''_{\alpha}\otimes L''_{0}$ is $P\circ \Delta''_{0,\alpha}$, where $P$ is the permutation, 
\item $\Delta''_{\alpha,\beta} \colon L''_{\alpha+\beta}\rightarrow L''_{\alpha}\otimes L''_{\beta}$ is 
\begin{align*}
\Delta''_{\alpha,\beta} (y_{\alpha+\beta})=\Delta''_{\alpha,\beta} (z_{\alpha+\beta})=&0, 
\end{align*}
\item $\varepsilon'' \colon L''_{0}\rightarrow \mathbf{F_{2}}$ is 
\begin{align*}
\varepsilon'' (1)=&0, \\
\varepsilon'' (x)=&1. 
\end{align*}
\end{itemize}
\end{rem}
{\bf Acknowledgements: } The author would like to thank Hitoshi Murakami and T{\' a}m{\' a}s Kalm{\' a}n for their encouragements and helpful comments. 
Jae Choon Cha gave him some useful comments about Remark~$\ref{general_group}$ at The $9$th East Asian School of Knots and Related Topics. 
He is also grateful to Jae Choon Cha. 
He also would like to thank the referee. 
He was supported by JSPS KAKENHI Grant Number 25001362. 
%
\bibliographystyle{amsplain}
\bibliography{mrabbrev,tagami}
\end{document}